\documentclass[reqno,11pt]{amsart}

\usepackage{amsmath,amssymb,amsthm}
\usepackage{geometry}
\usepackage[dvipsnames]{xcolor}
\usepackage{aliascnt}
\usepackage{hyperref}
\hypersetup{colorlinks=true, citecolor=blue, linkcolor=BrickRed, urlcolor=OliveGreen, pdfstartview={FitH}}

\newtheorem{theorem}{Theorem}[section]

\newaliascnt{proposition}{theorem}
\newtheorem{proposition}[proposition]{Proposition}
\aliascntresetthe{proposition}

\newaliascnt{lemma}{theorem}
\newtheorem{lemma}[lemma]{Lemma}
\aliascntresetthe{lemma}

\newaliascnt{corollary}{theorem}
\newtheorem{corollary}[corollary]{Corollary}
\aliascntresetthe{corollary}

\theoremstyle{remark}
\newaliascnt{remark}{theorem}
\newtheorem{remark}[remark]{Remark}
\aliascntresetthe{remark}

\usepackage[capitalise,noabbrev]{cleveref}

\crefname{theorem}{Theorem}{Theorems}
\Crefname{theorem}{Theorem}{Theorems}
\crefname{lemma}{Lemma}{Lemmas}
\Crefname{lemma}{Lemma}{Lemmas}
\crefname{corollary}{Corollary}{Corollaries}
\Crefname{corollary}{Corollary}{Corollaries}
\crefname{proposition}{Proposition}{Propositions}
\Crefname{proposition}{Proposition}{Propositions}
\crefname{remark}{Remark}{Remarks}
\Crefname{remark}{Remark}{Remarks}
\crefname{section}{Section}{Sections}
\Crefname{section}{Section}{Sections}

\usepackage{booktabs}

\usepackage{listings}
\DeclareMathOperator{\ord}{ord}
\DeclareMathOperator{\rad}{rad}
\DeclareMathOperator{\lcm}{lcm}

\newcommand{\N}{\mathbb{N}}
\newcommand{\Z}{\mathbb{Z}}
\newcommand{\R}{\mathbb{R}}
\newcommand{\vpp}[1]{\nu_{#1}}
\newcommand{\Pplus}{P^+}
\newcommand{\fpart}[1]{\left\{\!\left\{#1\right\}\!\right\}}
\newcommand{\defeq}{:=}
\newcommand{\ldefeq}{=:}
\newcommand{\term}[1]{\textit{#1}}

\title[Reciprocals in Missing-Digit Sets]{A Fixed-Prime Criterion for\\Reciprocals in Missing-Digit Sets}
\author[Scott Duke Kominers]{Scott Duke Kominers}
\address{Harvard Business School; Department of Economics and Center of Mathematical Sciences and Applications, Harvard University; and a16z crypto}
\email{kominers@fas.harvard.edu}
\thanks{I used LLMs to assist with some of the computations and coding in the preparation of this article, particularly GPT-5.4 Pro and Claude 4.6 Opus (both accessed via Poe with the support of Quora, where I am an advisor). I particularly thank Refine.ink for a suggestion that led to a substantial expansion of the scope of the main result, including the application presented in \cref{sec:structural-example}. The problem, methods, and eventual written form are my own; and of course any errors remain my responsibility. This work was conducted while I was visiting the Technological Innovation, Entrepreneurship, and Strategic Management (TIES) Group at the MIT Sloan School of Management; I greatly appreciates their hospitality.}


\begin{document}

\begin{abstract}
We prove a structural upper bound on the \(p\)-adic valuation of denominators of rationals belonging to a missing-digit set \(K_{m,D}\), generalizing a key step in recent work of Lin, Wu, and Yang~\cite{LWY} on reciprocals of factorials. For a rational \(\frac{r}{Q}\) with \(\gcd(Q,m)=1\) and a fixed prime \(p_0\nmid m\), membership in \(K_{m,D}\) forces \(\vpp{p_0}(Q)\) to be controlled by the \(p_0\)-adic valuation of the multiplicative order of \(m\) modulo the radical of \(Q\), with explicit overhead depending only on \(m\) and \(D\). Because the obstruction is stated at the level of a single denominator, a pair of sequence-specific valuation estimates converts it into an effective finiteness criterion for \(\left\{\frac{1}{a_n}:n\in\N\right\}\cap K_{m,D}\). Specializing to the case in which \(Q\) is the part of \(n!\) coprime to \(m\) recovers the fixed-prime step in the Lin--Wu--Yang argument. As applications, we treat reciprocals of superfactorials, products of polynomial values, and products of Fibonacci numbers. We also exhibit an exponential family---products of \((m^k-1)\)---for which the full structural criterion applies but a coarser largest-prime-factor formulation does not.
\end{abstract}

\maketitle

\section{Introduction}

Let \(m\ge 3\), let \(D\subset\{0,1,\dots,m-1\}\) with \(1<\#D<m\), and let \(K_{m,D}\subset[0,1]\) be the associated missing-digit set defined by
\[
K_{m,D}\defeq \left\{\sum_{j=1}^\infty \frac{a_j}{m^j} : a_j\in D \text{ for all } j\ge 1\right\}.
\]

Questions about intersections of the form
\[
\left\{\frac{1}{a_n}:n\in\N\right\}\cap K_{m,D}
\]
are typically delicate. For reciprocals of factorials (\(a_n=n!\)), Lin, Wu, and Yang~\cite{LWY} proved finiteness for every missing-digit set, and in the middle-third Cantor set \(C=K_{3,\{0,2\}}\) they determined the intersection exactly, answering a question posed by Jiang, Kong, Li, and Wang~\cite{Jiangetal}:
\begin{equation}
\left\{\frac{1}{n!}:n\in\N\right\}\cap C=\left\{\frac{1}{1!},\frac{1}{5!}\right\}=\left\{1,\frac{1}{120}\right\}.
\label{eq:lwy-intersection}
\end{equation}

Lin, Wu, and Yang~\cite{LWY} proved finiteness for factorials using a fixed-prime valuation growth argument: for each prime \(p_0\nmid m\), membership in \(K_{m,D}\) forces a bound on \(\vpp{p_0}(n!)\) that eventually contradicts its known growth. The present paper extracts the mechanism into an abstract criterion independent of the sequence. This makes the method reusable for broad classes of reciprocal sequences. Moreover, the structural criterion can be stronger than the coarser largest-prime-factor packaging that suffices in many natural examples.

The two key ingredients are Korobov's \cite{Korobov} digit-distribution estimate for purely periodic base-\(m\) expansions and the standard order-lifting formulas at prime powers. For the \(m\)-coprime part \(Q\) of a denominator and a fixed prime \(p_0\nmid m\), the relevant arithmetic cost is
\[
\vpp{p_0}\bigl(\ord_{\rad(Q)}(m)\bigr),
\]
which measures how much \(p_0\)-power divisibility the radical of \(Q\) can force into the period length; our main structural theorem shows that membership in \(K_{m,D}\) forces \(\vpp{p_0}(Q)\) to be bounded by this radical-order term, up to explicit overhead \(t(m,p_0)+\log_{p_0}(c_{m,D})\). Once this structural inequality is available, finiteness for a sequence of reciprocals \(\left\{\frac{1}{a_n}:n\in\N\right\}\) follows from a lower bound on \(\vpp{p_0}(Q_n)\) together with any upper bound on the radical-order term. A convenient corollary directly generalizing Lin, Wu, and Yang's method~\cite{LWY} leverages the obstruction using only the largest prime factor of \(Q\).

For comparison with~\cite{LWY}, we write
\[
M_n\defeq \frac{n!}{(n!,m^\infty)},
\]
where \((A,m^\infty)\defeq \prod_{p\mid m} p^{\nu_p(A)}\) denotes the full \(m\)-part of \(A\).
The factorial proof in~\cite{LWY} has two distinct steps. First, membership of \(\frac{1}{n!}\) in \(K_{m,D}\) forces a Korobov-type comparison between the period length modulo \(M_n\) and the period length modulo \(\rad(M_n)\); in the present notation this is the content of \cref{lem:korobov}, which attains a slightly sharper \(D\)-dependent constant \(c_{m,D}\) in place of the uniform constant used in~\cite{LWY}. Second, one chooses a fixed prime \(p_0\nmid m\) and proves that \(\nu_{p_0}(M_n)\) grows faster than the \(p_0\)-adic contribution coming from \(\ord_{\rad(M_n)}(m)\). Our analysis separates these two steps cleanly: \cref{prop:denominator-obstruction} is a denominator-level statement for an arbitrary \(m\)-coprime denominator \(Q\), and \cref{thm:fixed-prime-criterion-structural} then converts a pair of sequence-specific estimates
\[
\nu_{p_0}(Q_n)\ge \alpha(n),
\qquad
\nu_{p_0}\!\bigl(\ord_{\rad(Q_n)}(m)\bigr)\le \gamma(n)
\]
into a finiteness theorem. Thus the factorial case is recovered by substituting \(Q_n=M_n\), while the same obstruction applies verbatim to other sequences.

Reciprocals of superfactorials, products of polynomial values, and products of Fibonacci numbers can all be handled by our largest-prime-factor corollary. Meanwhile, the family \(\prod_{k=1}^n (m^k-1)\) shows that the structural term can remain logarithmic even when the naive bound coming from \(P^+(Q_n)\) is exponential. In that sense, the full structural theorem can be materially easier to verify than the coarse \(P^+\)-packaging.

The use of Korobov's digit-distribution estimate to study rational numbers in missing-digit sets originates with Shparlinski~\cite{Shparlinski}, who obtained quantitative bounds on the number of rationals of the form \(\frac{a}{q^k}\) (with \(\gcd(q,m)=1\)) belonging to \(K_{m,D}\), sharpening an earlier qualitative finiteness result of Schleischitz~\cite{Schleischitz}. The key observation of Lin, Wu, and Yang~\cite{LWY} was that the same Korobov-based comparison can be brought to bear on reciprocals of rapidly growing sequences such as \(n!\), where a fixed \(p\)-adic valuation of the denominator grows fast enough to overwhelm the \emph{constraint imposed by the radical}---more precisely, the corresponding prime-adic contribution to \(\ord_{\rad(Q_n)}(m)\), and hence in particular the largest-prime-factor bound used in that setting. Jiang, Kong, Li, and Wang~\cite{Jiangetal} studied the analogous intersection question for reciprocals of perfect squares by rather different methods; their approach exploits the arithmetic structure of a fixed quadratic denominator and does not require the rapid valuation growth on which~\cite{LWY} and the present criterion rely.

\section{Preliminaries}

This section collects the two main ingredients: a consequence of Korobov's digit-distribution estimate for base-\(m\) expansions of rationals (\cref{sec:korobov-estimate}), and the standard order-lifting lemmas for prime powers (\cref{sec:order-lifting}).

Maintaining the notation used in the Introduction, we fix a base \(m\ge 3\) and \(D\subset\{0,1,\dots,m-1\}\) with \(1<\#D<m\). The \term{missing-digit set} \(K_{m,D}\subset[0,1]\) is
\begin{equation}
K_{m,D}\defeq \left\{\sum_{j=1}^\infty \frac{a_j}{m^j} : a_j\in D \text{ for all } j\ge 1\right\}.
\label{eq:kmd-definition}
\end{equation}

\subsection{A consequence of Korobov's estimate}\label{sec:korobov-estimate}

For integers \(u,q\ge 1\) with \(\gcd(u,q)=1\), we write \(\ord_q(u)\) for the multiplicative order of \(u\) modulo \(q\), and we adopt the convention \(\ord_1(u)=1\). We write \(\rad(q)\) for the radical of \(q\).

\begin{lemma}\label{lem:korobov}
Set
\[
c_{m,D} \defeq 2(m-1)\min\left\{1,\frac{\# D}{m-\# D}\right\}
\]
and let \(0<r<q\) be integers with \(\gcd(mr,q)=1\).
If \(\frac{r}{q}\in K_{m,D}\), then
\[
\ord_q(m)\le c_{m,D}\cdot\ord_{\rad(q)}(m).
\]
\end{lemma}

\begin{proof}
Because \(\gcd(q,m)=1\), the reduced fraction \(\frac{r}{q}\) has a purely periodic base-\(m\) expansion. Equivalently, in the long-division algorithm the remainders stay in \((\Z/q\Z)^\times\), and the period length is
\[
\ord_q(m)\ldefeq \tau.
\]

For each digit \(j\in\{0,1,\dots,m-1\}\), let \(N_j\) be the number of occurrences of \(j\) in one full period in the purely periodic base-\(m\) expansion of \(\frac{r}{q}\). Then
\[
\sum_{j=0}^{m-1} N_j=\tau,
\]
and if \(\frac{r}{q}\in K_{m,D}\), then \(N_j=0\) for all \(j\notin D\).

Korobov's Theorem~1 \cite{Korobov} applies precisely in this coprime, purely periodic setting: for the period of the base-\(m\) expansion of a reduced fraction \(\frac{r}{q}\) with \(\gcd(q,m)=1\), it furnishes an auxiliary parameter \(\tau'=\tau'(m,q)\) such that
\begin{equation}
\left|N_j-\frac{\tau}{m}\right|
\le
\left(1-\frac{1}{m}\right)\tau'
\qquad
(j=0,1,\dots,m-1).
\label{eq:korobov-tau-prime}
\end{equation}
The point needed here is that Korobov's explicit formula~\cite[Equation~(9)]{Korobov} bounds \(\tau'\) only through the squarefree kernel of the denominator:
\begin{equation}
\tau'\le 2\,\ord_{\rad(q)}(m).
\label{eq:korobov-radical-bound}
\end{equation}
Thus the dependence on the denominator enters only through \(\rad(q)\), not through the exponents of its prime factors.

Choose \(i\notin D\). Since \(N_i=0\), \eqref{eq:korobov-tau-prime} gives
\[
\frac{\tau}{m}\le \left(1-\frac{1}{m}\right)\tau',
\]
hence, we have
\begin{equation}
\tau\le (m-1)\tau'.
\label{eq:korobov-tau-bound-one}
\end{equation}
Also, summing the upper bound \eqref{eq:korobov-tau-prime} over \(j\in D\), we obtain
\[
\tau=\sum_{j\in D} N_j
\le
\#D\cdot \frac{\tau}{m}+\#D\cdot\left(1-\frac{1}{m}\right)\tau',
\]
leading to
\[
\tau\left(1-\frac{\# D}{m}\right)\le \# D\cdot\left(1-\frac{1}{m}\right)\tau';
\]
and thus,
\begin{equation}
\tau\le \frac{\#D(m-1)}{m-\#D}\,\tau'.
\label{eq:korobov-tau-bound-two}
\end{equation}
Combining \eqref{eq:korobov-tau-bound-one} and \eqref{eq:korobov-tau-bound-two} yields
\[
\tau\le (m-1)\min\left\{1,\frac{\#D}{m-\#D}\right\}\tau'
=
\frac{c_{m,D}}{2}\tau'.
\]
Using \eqref{eq:korobov-radical-bound} then gives
\[
\ord_q(m)=\tau\le c_{m,D}\cdot\ord_{\rad(q)}(m).
\qedhere
\]
\end{proof}

\begin{remark}
Since \(m\ge 3\) and \(1<\#D<m\), we always have \(c_{m,D}\ge 4\).
When \(\#D<\frac{m}{2}\), the refined constant
\[
c_{m,D}=2(m-1)\frac{\#D}{m-\#D}
\]
is strictly smaller than \(2(m-1)\).
\end{remark}

\subsection{Order lifting}\label{sec:order-lifting}

We now record the order-lifting facts needed for the fixed-prime comparison.

\begin{lemma}[Odd prime powers]\label{lem:odd-prime-powers}
Let \(p\) be an odd prime and fix \(a\in\Z\) with \(\gcd(a,p)=1\).
Set \(d=\ord_p(a)\) and \(t=\vpp{p}(a^d-1)\).
Then for every \(k\ge 1\),
\[
\ord_{p^k}(a)=
\begin{cases}
d, & 1\le k\le t,\\
d\,p^{k-t}, & k>t.
\end{cases}
\]
In particular,
\[
\vpp{p}\bigl(\ord_{p^k}(a)\bigr)\ge \max\{0,k-t\}.
\]
\end{lemma}

\begin{proof}
This is standard; see for example \cite[Lemma~3]{Bloshchitsyn} or \cite[Equation~(2.5)]{LWY}.
Since \(d=\ord_p(a)\mid (p-1)\), we have \(p\nmid d\), and the valuation statement follows immediately from the displayed formula.
\end{proof}

The analogous statement for \(p_0=2\) requires a slightly different formulation because \((\Z/2^k\Z)^\times\) is not cyclic for \(k\ge 3\).

\begin{lemma}[\(2\)-adic overhead]\label{lem:two-adic-overhead}
Let \(a\) be an odd integer with \(a\neq \pm 1\), and define
\[
t_2(a)\defeq \vpp{2}(a^2-1)-1=\vpp{2}(a-1)+\vpp{2}(a+1)-1.
\]
Then for every \(k\ge 1\),
\[
\vpp{2}\bigl(\ord_{2^k}(a)\bigr)\ge \max\{0,k-t_2(a)\}.
\]
Moreover, we have
\begin{equation}
t_2(a)=\max\{\vpp{2}(a-1),\vpp{2}(a+1)\}.
\label{eq:t2-maximum}
\end{equation}
\end{lemma}

\begin{proof}
For every \(j\ge 1\), the \(2\)-adic lifting-the-exponent identity gives
\[
\vpp{2}(a^{2^j}-1)
=
\vpp{2}(a-1)+\vpp{2}(a+1)+j-1
=
t_2(a)+j.
\]
Now \(\ord_{2^k}(a)\) is a power of \(2\), say \(2^s\), because \((\Z/2^k\Z)^\times\) is a \(2\)-group.
Then \(a^{2^s}\equiv 1\pmod{2^k}\), so
\begin{equation}
\vpp{2}(a^{2^s}-1)\ge k.
\label{eq:two-adic-valuation-bound}
\end{equation}
If \(s\ge 1\), then \eqref{eq:two-adic-valuation-bound} yields
\[
t_2(a)+s\ge k;
\]
hence, \(s\ge k-t_2(a)\).
If \(s=0\), the stated bound is automatic.
Therefore
\[
\vpp{2}\bigl(\ord_{2^k}(a)\bigr)=s\ge \max\{0,k-t_2(a)\}.
\]

Finally, \((a-1)/2\) and \((a+1)/2\) are consecutive integers, so exactly one of them is even.
Thus
\[
\min\{\vpp{2}(a-1),\vpp{2}(a+1)\}=1,
\]
and \(t_2(a)\) reduces to the maximum~\eqref{eq:t2-maximum}.
\end{proof}

\section{A denominator obstruction and the sequence criterion}

Our main argument proceeds in three steps:
\begin{enumerate}
\item First, we reduce membership in \(K_{m,D}\) to a statement about the denominator after removing its \(m\)-part (\cref{sec:reduction}).
\item Next, we obtain a structural obstruction for a single denominator by combining Korobov's estimate with order-lifting (\cref{sec:single-denominator}).
\item Then, we convert the single-denominator obstruction into a sequence criterion and derive our largest-prime-factor corollary (\cref{sec:main-result}).
\end{enumerate}

\subsection{Reduction to an \(m\)-coprime denominator}\label{sec:reduction}

\begin{lemma}\label{lem:reduction}
Let \(A\ge 1\), and put
\[
Q\defeq \frac{A}{(A,m^\infty)}.
\]
Then \(\gcd(Q,m)=1\) by construction.
Assume \(Q\ge 2\).
If \(\frac{1}{A}\in K_{m,D}\), then there exists an integer \(r\) with
\[
0<r<Q,
\qquad
\gcd(mr,Q)=1,
\qquad
\frac{r}{Q}\in K_{m,D}.
\]
\end{lemma}

\begin{proof}
Since \(Q>1\) by hypothesis, the reduced denominator \(A\) of \(\frac{1}{A}\) has a prime factor not dividing \(m\); hence, \(\frac{1}{A}\) does not have a terminating base-\(m\) expansion.

We choose \(k\ge 0\) such that \((A,m^\infty)\mid m^k\), and define
\[
B\defeq \frac{m^k}{(A,m^\infty)}\in\N.
\]
Because \(Q\) is coprime to \(m\), we have
\[
\frac{m^k}{A}=\frac{B}{Q},
\qquad
\gcd(B,Q)=1.
\]

Let \(T:[0,1)\to[0,1)\) be the map
\[
T(x)=\fpart{mx},
\]
where \(\fpart{\cdot}\) denotes fractional part.
If \(x\in[0,1)\) has a nonterminating base-\(m\) expansion with all digits in \(D\), then \(T^k(x)\in K_{m,D}\) for every \(k\ge 0\).\footnote{We shall use \(T\) only for rationals whose reduced denominator does not divide any power of \(m\), so no ambiguity of base-\(m\) expansion arises.}

By shift-invariance of nonterminating digit expansions inside \(K_{m,D}\), we have
\[
T^k\left(\frac{1}{A}\right)=\fpart{\frac{m^k}{A}}\in K_{m,D}.
\]
Let \(r\) be the least positive residue of \(B\) modulo \(Q\).
Since \(\gcd(B,Q)=1\), we have \(0<r<Q\) and \(\gcd(r,Q)=1\), and
\[
\fpart{\frac{B}{Q}}=\frac{r}{Q}.
\]
Thus \(\frac{r}{Q}\in K_{m,D}\).
Finally, \(\gcd(mr,Q)=1\) because both \(\gcd(m,Q)=1\) and \(\gcd(r,Q)=1\) hold.
\end{proof}

\begin{remark}
If \(Q=1\), then the reduced denominator of \(\frac{1}{A}\) divides a power of \(m\), so \(\frac{1}{A}\) lies in the terminating regime rather than the purely periodic regime.
As we see in the sequel, the Korobov argument applies only after passing to the \(m\)-coprime denominator \(Q\ge 2\).
\end{remark}

\subsection{A single-denominator obstruction}\label{sec:single-denominator}

With the reduction to the \(m\)-coprime part in hand, we can state the key arithmetic obstruction.

Fix a prime \(p_0\) with \(\gcd(p_0,m)=1\), and define the overhead parameter \(t=t(m,p_0)\) by
\begin{equation}
t\defeq
\begin{cases}
\vpp{p_0}\bigl(m^{\ord_{p_0}(m)}-1\bigr) & p_0 \text{ odd},\\
\vpp{2}(m^2-1)-1 & p_0=2.
\end{cases}
\label{eq:overhead-parameter}
\end{equation}

\begin{proposition}[Structural denominator obstruction]\label{prop:denominator-obstruction}
Let \(p_0\nmid m\) be a prime, and let \(t=t(m,p_0)\) be as in \eqref{eq:overhead-parameter}.
Let \(Q\ge 2\) be an integer with \(\gcd(Q,m)=1\), and let \(r\) be an integer with \(0<r<Q\) and \(\gcd(mr,Q)=1\).
If \(\frac{r}{Q}\in K_{m,D}\), then
\[
\vpp{p_0}(Q)
\le
t+\vpp{p_0}\bigl(\ord_{\rad(Q)}(m)\bigr)+\log_{p_0}(c_{m,D}).
\]
Equivalently, we have
\[
\vpp{p_0}(Q)
\le
t+\max_{p\mid Q}\left\{\vpp{p_0}\bigl(\ord_p(m)\bigr)\right\}+\log_{p_0}(c_{m,D}).\footnote{Here and hereafter, \(p\mid\cdot\) is interpreted as indexing over prime numbers \(p\).}
\]
\end{proposition}

\begin{proof}
We set
\[
\tau_Q\defeq \ord_Q(m),
\qquad
\tau_*\defeq \ord_{\rad(Q)}(m).
\]
As \(\rad(Q)\mid Q\), we have \(\tau_*\mid \tau_Q\), so
\[
L\defeq \frac{\tau_Q}{\tau_*}\in\N.
\]
By \cref{lem:korobov}, we have
\[
L\le c_{m,D};
\]
hence, we obtain
\[
\vpp{p_0}(L)\le \log_{p_0}(c_{m,D}).
\]

Now let \(k=\vpp{p_0}(Q)\). If \(k=0\), then the desired inequality is immediate, since the right side is nonnegative. Thus, we assume \(k\ge 1\). Since \(p_0^k\mid Q\), we have
\[
\ord_{p_0^k}(m)\mid \ord_Q(m)=\tau_Q.
\]
Therefore---by \cref{lem:odd-prime-powers} when \(p_0\) is odd and by \cref{lem:two-adic-overhead} when \(p_0=2\)---we have
\[
\vpp{p_0}(\tau_Q)\ge \vpp{p_0}\bigl(\ord_{p_0^k}(m)\bigr)\ge k-t.
\]

Finally, we note that
\[
k-t
\le
\vpp{p_0}(\tau_Q)
=
\vpp{p_0}(L)+\vpp{p_0}(\tau_*)
\le
\log_{p_0}(c_{m,D})+\vpp{p_0}(\tau_*),
\]
and \(\tau_*=\ord_{\rad(Q)}(m)\), so the claim follows.
\end{proof}

\begin{remark}[Comparison with~\cite{LWY}]
Specializing \cref{prop:denominator-obstruction} to
\[
Q=M_n=\frac{n!}{(n!,m^\infty)}
\]
corresponds to the fixed-prime inequality that underlies the factorial argument of~\cite{LWY}. The new point is that \cref{prop:denominator-obstruction} is stated for an arbitrary \(m\)-coprime denominator \(Q\), with no factorial structure assumed; all sequence-dependent input is deferred to \cref{thm:fixed-prime-criterion-structural}.
\end{remark}

\subsection{Main result}\label{sec:main-result}

Combining the denominator obstruction (\cref{prop:denominator-obstruction}) with our reduction lemma (\cref{lem:reduction}) yields our criterion.

\begin{theorem}[Fixed-prime criterion: structural form]\label{thm:fixed-prime-criterion-structural}
Let \(m\ge 3\) and let \(K_{m,D}\) be a missing-digit set as defined in \eqref{eq:kmd-definition}.
For \(\{a_n\}_{n\ge 1}\subset\N\), define
\[
Q_n\defeq \frac{a_n}{(a_n,m^\infty)}.
\]
Fix a prime \(p_0\nmid m\), and let \(t=t(m,p_0)\) be as in \eqref{eq:overhead-parameter}.

If there exist explicit functions \(\alpha,\gamma:\N\to\R\) and an index \(N\in\N\) such that for every \(n\ge N\),
\[
\vpp{p_0}(Q_n)\ge \alpha(n),
\qquad
\vpp{p_0}\bigl(\ord_{\rad(Q_n)}(m)\bigr)\le \gamma(n),
\]
and moreover there exists \(n_0\ge N\) such that for all \(n\ge n_0\),
\begin{equation}
\alpha(n)>t+\gamma(n)+\log_{p_0}(c_{m,D}),
\label{eq:structural-cutoff}
\end{equation}
then we have
\[
\frac{1}{a_n}\notin K_{m,D}
\qquad
(n\ge n_0).
\]
In particular, the intersection
\[
\left\{\frac{1}{a_n}:n\in\N\right\}\cap K_{m,D}
\]
is finite.
\end{theorem}

\begin{proof}
We fix \(n\ge n_0\), and suppose for the sake of seeking a contradiction that \(\frac{1}{a_n}\in K_{m,D}\).
Since \eqref{eq:structural-cutoff} implies that \(\alpha(n)>t\), we have \(\vpp{p_0}(Q_n)\ge 1\); hence \(Q_n\ge 2\).
By \cref{lem:reduction}, there exists \(r_n\) such that
\[
0<r_n<Q_n,
\qquad
\gcd(mr_n,Q_n)=1,
\qquad
\frac{r_n}{Q_n}\in K_{m,D}.
\]
Applying \cref{prop:denominator-obstruction} with \(Q=Q_n\), we obtain
\[
\vpp{p_0}(Q_n)
\le
t+\vpp{p_0}\bigl(\ord_{\rad(Q_n)}(m)\bigr)+\log_{p_0}(c_{m,D}).
\]
Using our hypotheses then gives
\[
\alpha(n)\le t+\gamma(n)+\log_{p_0}(c_{m,D}),
\]
contradicting \eqref{eq:structural-cutoff}.
\end{proof}

Writing \(\Pplus(n)\) for the largest prime divisor of \(n\), with the convention that \(\Pplus(1)=1\), we obtain the following corollary.

\begin{corollary}[Fixed-prime criterion: largest-prime-factor form]\label{cor:fixed-prime-criterion-largest-prime}
Let \(m\ge 3\), let \(K_{m,D}\) be a missing-digit set, and take \(Q_n=\frac{a_n}{(a_n,m^\infty)}\) for a sequence \(\{a_n\}\subset\N\) as in \cref{thm:fixed-prime-criterion-structural}.
Fix a prime \(p_0\nmid m\), and let \(t=t(m,p_0)\) be as in \eqref{eq:overhead-parameter}.

If there exist explicit functions \(\alpha,\beta:\N\to\R\) and an index \(N\in\N\) such that for every \(n\ge N\),
\[
\vpp{p_0}(Q_n)\ge \alpha(n),
\qquad
\Pplus(Q_n)\le \beta(n),
\qquad
\beta(n)\ge 2,
\]
and moreover there exists \(n_0\ge N\) such that for all \(n\ge n_0\),
\begin{equation}
\alpha(n)
>
t+\log_{p_0}(\beta(n)-1)+\log_{p_0}(c_{m,D}),
\label{eq:largest-prime-cutoff}
\end{equation}
then we have
\[
\frac{1}{a_n}\notin K_{m,D}
\qquad
(n\ge n_0).
\]
In particular, the intersection
\[
\left\{\frac{1}{a_n}:n\in\N\right\}\cap K_{m,D}
\]
is finite.
\end{corollary}

\begin{proof}
For each \(n\), if \(Q_n=1\), then
\[
\vpp{p_0}\bigl(\ord_{\rad(Q_n)}(m)\bigr)
=
\vpp{p_0}\bigl(\ord_1(m)\bigr)
=
0
\le
\log_{p_0}(\beta(n)-1).
\]
If \(Q_n\ge 2\), then
\[
\vpp{p_0}\bigl(\ord_{\rad(Q_n)}(m)\bigr)
=
\max_{p\mid Q_n}\left\{\vpp{p_0}\bigl(\ord_p(m)\bigr)\right\}.
\]
For each prime \(p\mid Q_n\), we have \(\ord_p(m)\mid (p-1)\), hence
\[
\vpp{p_0}\bigl(\ord_p(m)\bigr)\le \vpp{p_0}(p-1)\le \log_{p_0}(p-1)\le \log_{p_0}(\beta(n)-1),
\]
because \(p\le \Pplus(Q_n)\le \beta(n)\). Thus
\[
\vpp{p_0}\bigl(\ord_{\rad(Q_n)}(m)\bigr)\le \log_{p_0}(\beta(n)-1),
\]
and the result follows from \cref{thm:fixed-prime-criterion-structural}.
\end{proof}

\begin{remark}
If
\[
G(n)\defeq \alpha(n)-\gamma(n)
\]
is eventually nondecreasing, then any verified tail on which
\[
G(n)>t+\log_{p_0}(c_{m,D})
\]
holds yields an explicit cutoff in \cref{thm:fixed-prime-criterion-structural}.

In the setting of \cref{cor:fixed-prime-criterion-largest-prime}, this specializes to
\[
F(n)\defeq \alpha(n)-\log_{p_0}(\beta(n)-1).
\]
When \(\alpha(n)\) is integer-valued, it may be convenient to replace \eqref{eq:largest-prime-cutoff} by the equivalent inequality
\begin{equation}
p_0^{\alpha(n)-t}>c_{m,D}\,(\beta(n)-1).
\label{eq:largest-prime-cutoff-exponential}
\end{equation}
\end{remark}

\section{Applications}\label{sec:applications}

We now turn to applications. The first four families lie in the regime where \cref{cor:fixed-prime-criterion-largest-prime} is easy to deploy; they may be viewed as progressively broader variants of the original Lin--Wu--Yang~\cite{LWY} setting. We then conclude with a more structural example, where \cref{thm:fixed-prime-criterion-structural} applies even though the coarse largest-prime-factor packaging need not.

Before treating those families, we record the exact remainder recurrence used in our finite verification steps, which is simply the base-\(m\) long-division algorithm written in a form convenient for reproducible membership tests.\footnote{In the few cases where \(\frac{1}{a_n}=1\), we handle membership separately via the nonterminating identity \(1=0.\overline{(m-1)}\) in base \(m\) (so, in particular, \(1=0.\overline{2}\in C\) in base \(3\)).}

\begin{lemma}[Exact base-\(m\) membership test away from the terminating regime]\label{lem:exact-base-m-test}
Let \(x=\frac{u}{v}\in(0,1)\) be in lowest terms, and assume that the denominator \(v\) does not divide any power of \(m\).
Define \(r_0=u\), and for \(j\ge 1\) write
\[
mr_{j-1}=d_jv+r_j,
\qquad
d_j\in\{0,1,\dots,m-1\},
\quad
0\le r_j<v.
\]
Then the unique base-\(m\) expansion of \(x\) is
\[
x=\sum_{j=1}^\infty \frac{d_j}{m^j},
\]
and the remainder sequence \((r_j)\) is eventually periodic. If \(k\) denotes the least positive integer for which \(r_k=r_i\) for some \(0\le i<k\), then
\[
x\in K_{m,D}
\iff
d_j\in D \text{ for every } 1\le j\le k.
\]
\end{lemma}

\begin{proof}
The recursion is exactly the base-\(m\) long-division algorithm: at step \(j\), the quotient digit is
\[
d_j=\left\lfloor \frac{mr_{j-1}}{v}\right\rfloor,
\]
and the new remainder is \(r_j=mr_{j-1}-d_jv\), so \(0\le r_j<v\). Iterating gives
\[
\frac{u}{v}
=
\sum_{j=1}^N \frac{d_j}{m^j}+\frac{r_N}{vm^N}
\qquad (N\ge 1),
\]
and letting \(N\to\infty\) yields
\[
x=\sum_{j=1}^\infty \frac{d_j}{m^j}.
\]

Since there are only finitely many possible remainders, there exists a least positive integer \(k\) such that \(r_k=r_i\) for some \(0\le i<k\). Hence \((r_j)\) is eventually periodic, and therefore so is \((d_j)\). A rational number has two base-\(m\) expansions only in the terminating case, i.e. when its denominator divides a power of \(m\); by hypothesis this does not occur, so the expansion here is unique.

Finally, once a remainder repeats, the long-division recursion becomes deterministic: if \(r_k=r_i\), then
\[
d_{k+\ell}=d_{i+\ell},
\qquad
r_{k+\ell}=r_{i+\ell}
\qquad (\ell\ge 1).
\]
Thus every digit after \(d_k\) is a repeat of an earlier digit among \(d_1,\dots,d_k\). Therefore
\[
x\in K_{m,D}
\iff
d_j\in D \text{ for every } 1\le j\le k.
\qedhere
\]
\end{proof}

\begin{remark}[Certification protocol for the finite checks]\label{rem:finite-checks}
Every exact intersection statement in this section is proved in two steps. First, a cutoff \(N_0\) is obtained from \cref{thm:fixed-prime-criterion-structural} or \cref{cor:fixed-prime-criterion-largest-prime}, so that \(\frac{1}{a_n}\notin K_{m,D}\) for all \(n\ge N_0\). Second, for each remaining \(1\le n<N_0\), we reduce \(\frac{1}{a_n}\) to lowest terms \(u_n/v_n\). If \(v_n=1\), membership is handled separately. If \(v_n>1\) but divides a power of \(m\), then \(\frac{u_n}{v_n}\) lies in the terminating regime and can again be handled directly. Otherwise \(v_n\) does not divide any power of \(m\), and \cref{lem:exact-base-m-test} decides membership exactly by checking the digits \(d_1,\dots,d_k\) up to the first repeated remainder \(r_k\). The tables in the sequel record these exact checks; throughout, the computation uses only exact arithmetic (arbitrary-precision integers, no floating point).\footnote{See \cref{app:computational-verification} for the implementation source code.}
\end{remark}

\subsection{Factorials}

Our criterion builds directly on the approach of Lin, Wu, and Yang~\cite{LWY}, and as such, we can directly recover their finiteness result for reciprocals of factorials in missing-digit sets.

\begin{proposition}[{Reciprocals of factorials~\cite[Theorem 1.4]{LWY}}]\label{prop:factorials}
Fix \(m\ge 3\) and a missing-digit set \(K_{m,D}\).
Then, the intersection
\[
\left\{\frac{1}{n!}:n\in\N\right\}\cap K_{m,D}
\]
is finite.

More precisely, for every prime \(p_0\nmid m\), \cref{cor:fixed-prime-criterion-largest-prime} applies for \(a_n=n!\), with
\[
\alpha(n)=\frac{n-s_{p_0}(n)}{p_0-1},
\qquad
\beta(n)=n,
\]
where \(s_{p_0}(n)\) denotes the sum of the base-\(p_0\) digits of \(n\).
\end{proposition}

\begin{proof}
As \(p_0\nmid m\), removing the \(m\)-part of \(a_n=n!\) does not change the \(p_0\)-adic valuation:
\[
\vpp{p_0}(Q_n)=\vpp{p_0}\left(\frac{n!}{(n!,m^\infty)}\right)=\vpp{p_0}(n!).
\]
Also every prime divisor of \(Q_n\) is at most \(n\), so \(\Pplus(Q_n)\le n\).

By Legendre's formula \cite{legendre1830theorie},
\[
\vpp{p_0}(n!)=\frac{n-s_{p_0}(n)}{p_0-1}.
\]
Then, using
\[
s_{p_0}(n)\le (p_0-1)\bigl(\lfloor \log_{p_0}n\rfloor+1\bigr),
\]
we obtain
\[
\alpha(n)\ge \frac{n}{p_0-1}-\log_{p_0}n-1.
\]
Thus, \(\alpha(n)\) grows linearly while the right side of \eqref{eq:largest-prime-cutoff} is \(O(\log n)\) in this instance, so \cref{cor:fixed-prime-criterion-largest-prime} applies with \(\beta(n)=n\).
\end{proof}

\begin{remark}\label{rem:factorial-explicit}
For \(m=3\) and either \(D=\{0,1\}\) or \(D=\{0,2\}\), we have \(c_{3,D}=4\); if we choose \(p_0=2\), then \(t=t_2(3)=2\).
Using the exact identity \(\vpp{2}(n!)=n-s_2(n)\), we can verify that the cutoff inequality \eqref{eq:largest-prime-cutoff} already holds for all \(n\ge 10\), which is slightly sharper than the \(n\ge 21\) cutoff obtained by Lin, Wu, and Yang~\cite{LWY}. In the middle-third Cantor case \(D=\{0,2\}\), applying \cref{lem:exact-base-m-test} to each reciprocal \(\frac{1}{n!}\) for \(1\le n\le 9\) then recovers \eqref{eq:lwy-intersection}; see \cref{tab:factorials}.
\end{remark}

\begin{table}[tp]
\centering

\renewcommand{\arraystretch}{1.15}
\begin{tabular}{rlr}
\toprule
$n$ & base-$3$ prefix of $\tfrac{1}{a_n}$ & first $1$ pos.\\
\midrule
$1$ & $0.\overline{2}$ & --- \\
$2$ & $0.\mathbf{1}11\ldots$ & $1$ \\
$3$ & $0.0\mathbf{1}11\ldots$ & $2$ \\
$4$ & $0.00\mathbf{1}01\ldots$ & $3$ \\
$5$ & $0.0\overline{0002}$ & --- \\
$6$ & $0.00000\mathbf{1}00\ldots$ & $6$ \\
$7$ & $0.0000000\mathbf{1}02\ldots$ & $8$ \\
$8$ & $0.000000000\mathbf{1}11\ldots$ & $10$ \\
$9$ & $0.00000000000\mathbf{1}11\ldots$ & $12$ \\
\bottomrule
\end{tabular}

\vspace*{12pt}
\caption{Ternary verification for reciprocal factorials \(a_n=n!\), \(1\le n\le 9\).
The column ``first~\(1\) pos.''\ gives the position of the first ternary digit equal to \(1\) in the base-\(3\) expansion of \(\frac{1}{a_n}\); a dash indicates that \(\frac{1}{a_n}\in C\).}
\label{tab:factorials}
\end{table}

\subsection{Superfactorials}

Summing the factorial valuation estimate over \(1\le k\le n\) immediately yields a quadratic-growth variant for superfactorials.

\begin{proposition}[Reciprocals of superfactorials]\label{prop:superfactorials}
Fix \(m\ge 3\) and a missing-digit set \(K_{m,D}\), and let
\[
a_n=\prod_{k=1}^n k!.
\]
Then, the intersection \(\left\{\frac{1}{a_n}:n\in\N\right\}\cap K_{m,D}\) is finite.
\end{proposition}

\begin{proof}
We fix a prime \(p_0\nmid m\).
As in the case of factorials (\cref{prop:factorials}), removing the \(m\)-part does not change the \(p_0\)-adic valuation of \(a_n\), so we obtain
\[
\vpp{p_0}(Q_n)=\sum_{k=1}^n \vpp{p_0}(k!);
\]
and again every prime divisor of \(Q_n\) is at most \(n\); hence, \(\Pplus(Q_n)\le n\).

By Legendre's formula and the bound \(s_{p_0}(k)\le (p_0-1)(\log_{p_0}k+1)\), we see that
\begin{equation}
\vpp{p_0}(k!)
=
\frac{k-s_{p_0}(k)}{p_0-1}
\ge
\frac{k}{p_0-1}-\log_{p_0}k-1.
\label{eq:superfactorial-factorial-bound}
\end{equation}
Summing \eqref{eq:superfactorial-factorial-bound} over \(1\le k\le n\), we obtain
\[
\vpp{p_0}(Q_n)
\ge
\frac{n(n+1)}{2(p_0-1)}-n\log_{p_0}n-n.
\]
Thus taking \(\alpha(n)=\vpp{p_0}(Q_n)\), we see that the left side of \eqref{eq:largest-prime-cutoff} grows quadratically, while the right side of \eqref{eq:largest-prime-cutoff} in this instance is again \(O(\log n)\) when \(\beta(n)=n\); thus, \cref{cor:fixed-prime-criterion-largest-prime} applies.
\end{proof}

As an illustration, we derive the explicit bound for the standard Cantor set.

\begin{corollary}[Reciprocals of superfactorials in the standard Cantor set]
Let \(C=K_{3,\{0,2\}}\), and define
\[
a_n=\prod_{k=1}^n k!.
\]
Then, we have
\[
\left\{\frac{1}{a_n}:n\in\N\right\}\cap C
=
\left\{1,\frac{1}{12}\right\};
\]
equivalently,
\[
\frac{1}{a_n}\in C
\iff
n\in\{1,3\}.
\]
\end{corollary}

\begin{proof}
Again \(m=3\), \(D=\{0,2\}\), \(c_{3,D}=4\), and for \(p_0=2\) we have \(t=2\).
As in the proof of \cref{prop:superfactorials}, \cref{cor:fixed-prime-criterion-largest-prime} applies to \(a_n=\prod_{k=1}^n k!\) with
\[
\alpha(n)=\sum_{k=1}^n \vpp{2}(k!)=\sum_{k=1}^n (k-s_2(k)),
\qquad
\beta(n)=n.
\]

A direct evaluation gives
\[
\alpha(5)=0+1+1+3+3=8,
\]
while we have
\[
t+\log_2(\beta(5)-1)+\log_2(c_{3,D})
=
2+\log_2(4)+\log_2(4)
=
6.
\]
Thus \eqref{eq:largest-prime-cutoff} holds at \(n=5\). Moreover,
\[
\alpha(n+1)-\alpha(n)=\vpp{2}((n+1)!)\ge 1
\qquad (n\ge 1),
\]
whereas
\begin{align*}
&\bigl(2+\log_2(\beta(n+1)-1)+\log_2(4)\bigr)-\bigl(2+\log_2(\beta(n)-1)+\log_2(4)\bigr)\\
&=
\bigl(2+\log_2 n+\log_2(4)\bigr)-\bigl(2+\log_2(n-1)+\log_2(4)\bigr)\\
&=
\log_2\left(\frac{n}{n-1}\right)\\
&<1
\qquad (n\ge 3).
\end{align*}
Hence, once \eqref{eq:largest-prime-cutoff} holds at \(n=5\), it holds for every \(n\ge 5\). Therefore
\[
\frac{1}{a_n}\notin C
\qquad
(n\ge 5).
\]

It remains to inspect \(1\le n\le 4\). The case \(n=1\) gives \(\frac{1}{a_1}=1\in C\). For \(n\ge 2\), the denominator \(a_n\) is even, so its reduced denominator does not divide any power of \(3\); thus, \cref{lem:exact-base-m-test} applies exactly. The check for \(2\le n\le 4\) shows that \(n=3\) (where \(a_3=1!\cdot 2!\cdot 3!=12\)) gives the only other member; see \cref{tab:superfactorials}. Thus
\[
\left\{\frac{1}{a_1},\frac{1}{a_3}\right\}=\left\{1,\frac{1}{12}\right\}\subset C,
\]
and for every other \(n\) with \(1\le n\le 4\), the reciprocal \(\frac{1}{a_n}\) has a ternary expansion containing the digit \(1\), and hence does not belong to \(C\).
\end{proof}

\begin{table}[tp]
\centering

\renewcommand{\arraystretch}{1.15}
\begin{tabular}{rlr}
\toprule
$n$ & base-$3$ prefix of $\tfrac{1}{a_n}$ & first $1$ pos.\\
\midrule
$1$ & $0.\overline{2}$ & --- \\
$2$ & $0.\mathbf{1}11\ldots$ & $1$ \\
$3$ & $0.0\overline{02}$ & --- \\
$4$ & $0.000002\mathbf{1}12\ldots$ & $7$ \\
\bottomrule
\end{tabular}

\vspace*{12pt}
\caption{Ternary verification for reciprocals of superfactorials \(a_n=\prod_{k=1}^n k!\), \(1\le n\le 4\).
Column conventions are as in \cref{tab:factorials}.}
\label{tab:superfactorials}
\end{table}

\subsection{Products of polynomial values}

The previous two examples remain close to the factorial setting. We now move to a broader class in which Legendre's formula is no longer available, but a single congruence class modulo a suitable prime still furnishes the required linear valuation growth.

\begin{proposition}[Reciprocals of products of polynomial values]\label{prop:polynomial-products}
Fix \(m\ge 3\) and a missing-digit set \(K_{m,D}\).
If \(f\in\Z[x]\) is a nonconstant polynomial of degree \(d\ge 1\) such that \(f(k)>0\) for all \(k\ge 1\), and
\[
a_n=\prod_{k=1}^n f(k),
\]
then \(\left\{\frac{1}{a_n}:n\in\N\right\}\cap K_{m,D}\) is finite.
\end{proposition}

\begin{proof}
By Schur's theorem on prime divisors of polynomial values (see \cite{Schur} or, e.g., \cite[Theorem 8.1]{Narkiewicz}), the set of primes dividing the values \(f(1),f(2),\dots\) is infinite.
Hence, there exists a prime \(p_0\nmid m\) and a residue class \(b\in\{0,1,\dots,p_0-1\}\) such that
\[
f(b)\equiv 0\pmod{p_0}.
\]

We write
\[
f(x)=c_dx^d+\cdots +c_0,
\qquad
C_f\defeq \sum_{i=0}^d |c_i|;
\]
for every \(k\ge 1\), we have
\[
0<f(k)\le C_f k^d.
\]
Thus, every prime divisor of
\[
Q_n=\frac{\prod_{k=1}^n f(k)}{\left(\prod_{k=1}^n f(k),m^\infty\right)}
\]
is at most \(C_f n^d\), so
\[
\Pplus(Q_n)\le \max\{2,C_f n^d\}.
\]

Since \(p_0\nmid m\), removing the \(m\)-part does not change the \(p_0\)-adic valuation:
\[
\vpp{p_0}(Q_n)=\sum_{k=1}^n \vpp{p_0}(f(k)).
\]
Among the integers \(1,\ldots,n\), each residue class modulo \(p_0\) occurs at least \(\lfloor \frac{n}{p_0}\rfloor\) times.
Hence, we have
\[
\vpp{p_0}(Q_n)\ge \#\{1\le k\le n:k\equiv b\!\!\!\pmod{p_0}\}\ge \left\lfloor\frac{n}{p_0}\right\rfloor.
\]
Thus, \cref{cor:fixed-prime-criterion-largest-prime} applies with
\[
\alpha(n)=\left\lfloor\frac{n}{p_0}\right\rfloor,
\qquad
\beta(n)=\max\{2,C_f n^d\}.
\]
As \(\alpha(n)\) grows linearly and \(\log_{p_0}\beta(n)=O(\log n)\), the cutoff inequality \eqref{eq:largest-prime-cutoff} holds for all sufficiently large \(n\).
\end{proof}

\begin{remark}[Effectivity of the auxiliary prime]\label{rem:polynomial-effective-prime}
For a given polynomial \(f\), the auxiliary prime \(p_0\) in \cref{prop:polynomial-products} can be found effectively by finite search: enumerate \(k=1,2,\dots\), factor \(f(k)\), and stop once a prime divisor \(p_0\nmid m\) is found. Schur's theorem guarantees that this procedure terminates and, once \(p_0\) is fixed, the cutoff furnished by \cref{cor:fixed-prime-criterion-largest-prime} is explicit.
\end{remark}

\Cref{prop:polynomial-products} logically generalizes \cref{prop:factorials}, which corresponds to the specialization \(f(k)=k\). More broadly, we can illustrate the result for other polynomials.

\begin{corollary}[Reciprocals of polynomial products: the case \(f(x)=x^2+1\)]
Let \(C=K_{3,\{0,2\}}\), and define
\[
a_n=\prod_{k=1}^n (k^2+1).
\]
Then, we have
\[
\left\{\frac{1}{a_n}:n\in\N\right\}\cap C
=
\left\{\frac{1}{10}\right\};
\]
equivalently,
\[
\frac{1}{a_n}\in C
\iff
n=2.
\]
\end{corollary}

\begin{proof}
We have \(m=3\), \(D=\{0,2\}\), \(c_{3,D}=4\), and we choose \(p_0=2\), so \(t=t_2(3)=2\).
For \(f(x)=x^2+1\), the proof of \cref{prop:polynomial-products} applies with \(b\equiv 1\pmod 2\), as
\[
f(1)=1^2+1=2\equiv 0\pmod 2.
\]

Here \(d=2\) and \(C_f=|1|+|0|+|1|=2\), so we may take
\[
\alpha(n)=\left\lfloor \frac{n}{2} \right\rfloor,
\qquad
\beta(n)=2n^2.
\]
Because \(k^2+1\not\equiv 0\pmod 3\) for any integer \(k\), every \(a_n\) is coprime to \(3\); in particular, every \(\frac{1}{a_n}\) lies outside the terminating regime.

Since \(t+\log_2(c_{3,D})=4\), the cutoff inequality \eqref{eq:largest-prime-cutoff} is equivalent to the form \eqref{eq:largest-prime-cutoff-exponential}:
\[
2^{\alpha(n)-4}>\beta(n)-1.
\]
At \(n=30\) and \(n=31\) we have
\begin{gather*}
\alpha(30)=15,\qquad \beta(30)-1=1799<2^{11},\\
\alpha(31)=15,\qquad \beta(31)-1=1921<2^{11},
\end{gather*}
respectively, so \eqref{eq:largest-prime-cutoff} holds in each case.

Moreover, for every \(n\ge 1\),
\[
\alpha(n+2)-\alpha(n)=1,
\]
while for every \(n\ge 5\),
\[
\log_2(\beta(n+2)-1)-\log_2(\beta(n)-1)
=
\log_2\!\left(\frac{2(n+2)^2-1}{2n^2-1}\right)
<
1.
\]
Hence, the quantity
\[
\alpha(n)-4-\log_2(\beta(n)-1)
\]
strictly increases when \(n\) is replaced by \(n+2\) for every \(n\ge 5\). Since the cutoff inequality holds at both \(n=30\) and \(n=31\), it follows by parity induction that it holds for every \(n\ge 30\). Consequently,
\[
\frac{1}{a_n}\notin C
\qquad
(n\ge 30).
\]

It remains to inspect the cases \(1\le n\le 29\). Applying \cref{lem:exact-base-m-test} to each \(\frac{1}{a_n}\) for \(1\le n\le 29\) (see \cref{tab:polynomial-products}) shows that
\[
\frac{1}{a_2}=\frac{1}{10}\in C,
\]
and for every other \(n\) with \(1\le n\le 29\), the reciprocal \(\frac{1}{a_n}\) does not lie in \(C\).
\end{proof}

\begin{table}[tp]
\centering

\renewcommand{\arraystretch}{1.15}
\begin{tabular}{rlr}
\toprule
$n$ & base-$3$ prefix of $\tfrac{1}{a_n}$ & first $1$ pos.\\
\midrule
$1$ & $0.\mathbf{1}11\ldots$ & $1$ \\
$2$ & $0.\overline{0022}$ & --- \\
$3$ & $0.00002\mathbf{1}02\ldots$ & $6$ \\
$4$ & $0.000000\mathbf{1}02\ldots$ & $7$ \\
$5$ & $0.000000000\mathbf{1}10\ldots$ & $10$ \\
$6$ & $0.00{\ldots}00\mathbf{1}21\ldots$ & $23$ \\
$7$ & $0.0000000000000000\mathbf{1}12\ldots$ & $17$ \\
$8$ & $0.00{\ldots}00\mathbf{1}22\ldots$ & $21$ \\
$9$ & $0.00{\ldots}00\mathbf{1}22\ldots$ & $25$ \\
$10$ & $0.00{\ldots}00\mathbf{1}12\ldots$ & $29$ \\
$11$ & $0.00{\ldots}00\mathbf{1}00\ldots$ & $33$ \\
$12$ & $0.00{\ldots}00\mathbf{1}20\ldots$ & $38$ \\
$13$ & $0.00{\ldots}02\mathbf{1}10\ldots$ & $44$ \\
$14$ & $0.00{\ldots}00\mathbf{1}00\ldots$ & $47$ \\
$15$ & $0.00{\ldots}00\mathbf{1}00\ldots$ & $52$ \\
$16$ & $0.00{\ldots}00\mathbf{1}00\ldots$ & $57$ \\
$17$ & $0.00{\ldots}02\mathbf{1}21\ldots$ & $64$ \\
$18$ & $0.00{\ldots}00\mathbf{1}22\ldots$ & $68$ \\
$19$ & $0.00{\ldots}00\mathbf{1}02\ldots$ & $73$ \\
$20$ & $0.00{\ldots}02\mathbf{1}01\ldots$ & $80$ \\
$21$ & $0.00{\ldots}00\mathbf{1}02\ldots$ & $84$ \\
$22$ & $0.00{\ldots}00\mathbf{1}22\ldots$ & $90$ \\
$23$ & $0.00{\ldots}00\mathbf{1}11\ldots$ & $108$ \\
$24$ & $0.00{\ldots}00\mathbf{1}01\ldots$ & $101$ \\
$25$ & $0.00{\ldots}00\mathbf{1}02\ldots$ & $107$ \\
$26$ & $0.00{\ldots}00\mathbf{1}10\ldots$ & $113$ \\
$27$ & $0.00{\ldots}00\mathbf{1}10\ldots$ & $119$ \\
$28$ & $0.00{\ldots}00\mathbf{1}02\ldots$ & $125$ \\
$29$ & $0.00{\ldots}00\mathbf{1}01\ldots$ & $131$ \\
\bottomrule
\end{tabular}

\vspace*{12pt}
\caption{Ternary verification for polynomial products \(a_n=\prod_{k=1}^n(k^2+1)\), \(1\le n\le 29\).
Column conventions are as in \cref{tab:factorials}.}
\label{tab:polynomial-products}
\end{table}

\subsection{Products of Fibonacci numbers}

The preceding examples all lie in a regime where the largest prime factor of the denominator is at most polynomial in \(n\). We next consider a sequence whose prime factors can be exponentially large, pushing \cref{cor:fixed-prime-criterion-largest-prime} into a genuinely marginal range.

Let \((F_k)_{k\ge 1}\) be the Fibonacci sequence defined by \(F_1=F_2=1\) and \(F_{k+2}=F_{k+1}+F_k\).

\begin{lemma}\label{lem:fibonacci-sum}
Let
\[
S(n)\defeq \sum_{k=1}^n \vpp{2}(F_k).
\]
Then for every \(n\ge 1\),
\begin{equation}
S(n)\ge
\left\lfloor\frac{n}{3}\right\rfloor
+
2\left\lfloor\frac{n}{6}\right\rfloor
+
\sum_{j=2}^{\lfloor \log_2(n/3)\rfloor}
\left\lfloor\frac{n}{3\cdot 2^j}\right\rfloor,
\label{eq:fibonacci-sum-lower-bound}
\end{equation}
where the final sum is empty if \(n<12\).
In particular, for \(n\ge 12\),
\[
S(n)\ge \frac{5}{6}n-\log_2 n-4.
\]
\end{lemma}

\begin{proof}
The main inequality \eqref{eq:fibonacci-sum-lower-bound} can be derived from Lengyel's explicit characterization of the \(2\)-adic valuation of Fibonacci numbers \cite[Lemma~2]{Lengyel}. We include a direct proof here for completeness.

We introduce the Lucas numbers \((L_k)_{k\ge 0}\) by \(L_0=2\), \(L_1=1\), and \(L_{k+2}=L_{k+1}+L_k\).
The doubling identity
\[
F_{2u}=F_uL_u
\]
is standard (see, e.g., \cite{Koshy}).

Reducing their defining recurrences modulo \(2\), we see that both \((F_k)\) and \((L_k)\) are periodic with period \(3\).
Hence, 
\[
\vpp{2}(F_k)\ge 1 \iff 3\mid k,
\qquad
\vpp{2}(L_k)\ge 1 \iff 3\mid k.
\]

We claim that for every \(j\ge 1\),
\[
\vpp{2}(F_{3\cdot 2^j})\ge j+2.
\]
Indeed, \(F_6=8\), so the claim holds for \(j=1\).
If it holds for some \(j\ge 1\), then
\[
F_{3\cdot 2^{j+1}}
=
F_{2(3\cdot 2^j)}
=
F_{3\cdot 2^j}L_{3\cdot 2^j},
\]
and \(3\mid 3\cdot 2^j\), so \(L_{3\cdot 2^j}\) is even.
Therefore, 
\[
\vpp{2}(F_{3\cdot 2^{j+1}})
\ge
\vpp{2}(F_{3\cdot 2^j})+1
\ge
(j+2)+1.
\]

Next, \(F_d\mid F_k\) whenever \(d\mid k\).
Hence, if \(3\cdot 2^j\mid k\), then
\[
\vpp{2}(F_k)\ge \vpp{2}(F_{3\cdot 2^j})\ge j+2.
\]
Therefore, pointwise in \(k\),
\[
\vpp{2}(F_k)
\ge
\mathbf{1}_{3\mid k}
+
2\,\mathbf{1}_{6\mid k}
+
\sum_{j=2}^\infty \mathbf{1}_{3\cdot 2^j\mid k}.
\]
Summing over \(1\le k\le n\) yields the stated lower bound for \(S(n)\).

Now assume \(n\ge 12\), and put \(J=\lfloor \log_2(\frac{n}{3})\rfloor\).
Using \(\lfloor x\rfloor\ge x-1\), we obtain
\begin{equation}
S(n)
\ge
\left(\frac{n}{3}-1\right)
+
2\left(\frac{n}{6}-1\right)
+
\sum_{j=2}^J\left(\frac{n}{3\cdot 2^j}-1\right).
\label{eq:fibonacci-sum-floor-bound}
\end{equation}
Since we have
\[
\sum_{j=2}^J \frac{1}{2^j}
=
\frac{1}{2}-\frac{1}{2^J},
\]
and \(2^{J+1}>n/3\), we obtain
\[
\frac{1}{2^J}\le \frac{6}{n}.
\]
Therefore, we find
\[
\sum_{j=2}^J \frac{1}{2^j}\ge \frac{1}{2}-\frac{6}{n};
\]
substituting this into \eqref{eq:fibonacci-sum-floor-bound} gives
\[
S(n)
\ge
\frac{2n}{3}-3
+
\frac{n}{3}\left(\frac{1}{2}-\frac{6}{n}\right)
-(J-1)
=
\frac{5n}{6}-J-4.
\]
Finally, we have \(J\le \log_2 n\), so
\[
S(n)\ge \frac{5}{6}n-\log_2 n-4
\]
as claimed.
\end{proof}

\begin{proposition}[Reciprocals of products of Fibonacci numbers]\label{prop:fibonacci-products}
Fix \(m\ge 3\) odd and a missing-digit set \(K_{m,D}\), and let
\[
a_n=\prod_{k=1}^n F_k.
\]
Then, the intersection \(\left\{\frac{1}{a_n}:n\in\N\right\}\cap K_{m,D}\) is finite.
\end{proposition}

\begin{proof}
We apply \cref{cor:fixed-prime-criterion-largest-prime} with \(p_0=2\).
As \(m\) is odd, removing the \(m\)-part removes only odd prime factors; therefore
\[
\vpp{2}(Q_n)=\vpp{2}(a_n)=\sum_{k=1}^n \vpp{2}(F_k).
\]
By \cref{lem:fibonacci-sum}, for \(n\ge 12\), we have
\[
\vpp{2}(Q_n)\ge \frac{5}{6}n-\log_2 n-4.
\]

Every prime divisor of \(Q_n\) is also a prime divisor of \(a_n\), and hence divides some \(F_k\) with \(k\le n\).
Writing \(\varphi=(1+\sqrt5)/2\), Binet's formula (see~\cite{Koshy}) gives \(F_n<\varphi^n\), so
\[
\Pplus(Q_n)\le F_n<\varphi^n.
\]
Thus, \cref{cor:fixed-prime-criterion-largest-prime} applies with
\[
\alpha(n)=\frac{5}{6}n-\log_2 n-4,
\qquad
\beta(n)=\varphi^n
\qquad
(n\ge 12).
\]
Here the cutoff inequality \eqref{eq:largest-prime-cutoff} holds for all sufficiently large \(n\) because
\[
\frac{5}{6}>\log_2\varphi.
\qedhere
\]
\end{proof}

As with our earlier applications, we can make the result completely explicit for the standard Cantor set.

\begin{corollary}[Reciprocals of products of Fibonacci numbers in the standard Cantor set]
For \(C=K_{3,\{0,2\}}\) and
\[
a_n=\prod_{k=1}^n F_k,
\]
(where \(F_k\) denotes the \(k\)-th Fibonacci number), we have
\[
\left\{\frac{1}{a_n}:n\in\N\right\}\cap C
=
\left\{1,\frac{1}{30}\right\};
\]
equivalently,
\[
\frac{1}{a_n}\in C
\iff
n\in\{1,2,5\}.
\]
\end{corollary}

\begin{proof}
We again work with \(m=3\) and \(D=\{0,2\}\), so \(c_{3,D}=4\). Choosing \(p_0=2\), we again have \(t=t_2(3)=2\).
Since \(m=3\) is odd, \cref{prop:fibonacci-products} applies to the sequence
\[
a_n=\prod_{k=1}^n F_k.
\]

Using the lower bound established in the proof of \cref{prop:fibonacci-products}, we may take
\begin{gather*}
\alpha(n)=\frac{5n}{6}-\log_2 n-4
\qquad (n\ge 12),\\
\beta(n)=\varphi^n,
\qquad
\varphi=\frac{1+\sqrt5}{2}.
\end{gather*}
Since \(\varphi^n-1<\varphi^n\), it is enough for \eqref{eq:largest-prime-cutoff} to prove
\[
\frac{5n}{6}-\log_2 n-4
>
2+n\log_2\varphi+\log_2(4),
\]
that is,
\[
\left(\frac{5}{6}-\log_2\varphi\right)n>\log_2 n+8.
\]
Now \(\log_2\varphi<0.6943\), so
\[
\delta\defeq \frac{5}{6}-\log_2\varphi > 0.1390.
\]
Consider
\[
h(x)\defeq \delta x-\log_2 x-8.
\]
Then
\[
h'(x)=\delta-\frac{1}{x\ln 2}>0
\qquad (x\ge 11),
\]
so \(h\) is increasing on \([11,\infty)\). Moreover,
\[
h(106)>0.1390\cdot 106-\log_2 106-8>14.734-6.728-8>0.
\]
Therefore \(h(n)>0\) for every \(n\ge 106\), and hence \eqref{eq:largest-prime-cutoff} holds for all \(n\ge 106\). Consequently,
\[
\frac{1}{a_n}\notin C
\qquad
(n\ge 106).
\]

It remains to check \(1\le n\le 105\). The cases \(n=1,2\) give \(\frac{1}{a_n}=1\in C\). For \(n\ge 3\), the factor \(F_3=2\) divides \(a_n\), so the reduced denominator of \(\frac{1}{a_n}\) does not divide any power of \(3\); thus, \cref{lem:exact-base-m-test} applies exactly on the remaining finite range. The check shows that
\[
a_5=F_1F_2F_3F_4F_5=1\cdot 1\cdot 2\cdot 3\cdot 5=30,
\]
so \(\frac{1}{a_5}=\frac{1}{30}\in C\), and that no other \(3\le n\le 105\) contributes; see \cref{tab:fibonacci-products}. Therefore
\[
\left\{\frac{1}{a_n}:n\in\N\right\}\cap C
=
\left\{1,\frac{1}{30}\right\},
\]
as claimed.
\end{proof}

\begin{table}[tp]
\centering
{\tiny
\renewcommand{\arraystretch}{1.05}
\[
\begin{array}{rlr @{\hspace{1.5em}\vrule width 0.4pt\hspace{1.5em}} rlr @{\hspace{1.5em}\vrule width 0.4pt\hspace{1.5em}} rlr}
n & \text{prefix of $1/a_n$} & \text{1st pos.} & n & \text{prefix of $1/a_n$} & \text{1st pos.} & n & \text{prefix of $1/a_n$} & \text{1st pos.} \\
\hline
1 & 0.\overline{2} & \text{--} & 36 & 0.00{\ldots}00\mathbf{1}12\ldots & 266 & 71 & 0.00{\ldots}00\mathbf{1}02\ldots & 1068 \\
2 & 0.\overline{2} & \text{--} & 37 & 0.00{\ldots}00\mathbf{1}20\ldots & 288 & 72 & 0.00{\ldots}00\mathbf{1}12\ldots & 1099 \\
3 & 0.\mathbf{1}11\ldots & 1 & 38 & 0.00{\ldots}00\mathbf{1}00\ldots & 297 & 73 & 0.00{\ldots}00\mathbf{1}02\ldots & 1130 \\
4 & 0.0\mathbf{1}11\ldots & 2 & 39 & 0.00{\ldots}22\mathbf{1}22\ldots & 326 & 74 & 0.00{\ldots}00\mathbf{1}20\ldots & 1162 \\
5 & 0.0\overline{0022} & \text{--} & 40 & 0.00{\ldots}22\mathbf{1}02\ldots & 333 & 75 & 0.00{\ldots}00\mathbf{1}11\ldots & 1194 \\
6 & 0.0000\mathbf{1}00\ldots & 5 & 41 & 0.00{\ldots}20\mathbf{1}20\ldots & 350 & 76 & 0.00{\ldots}02\mathbf{1}12\ldots & 1228 \\
7 & 0.000000020022\mathbf{1}00\ldots & 13 & 42 & 0.00{\ldots}00\mathbf{1}00\ldots & 365 & 77 & 0.00{\ldots}02\mathbf{1}12\ldots & 1261 \\
8 & 0.0000000000220\mathbf{1}00\ldots & 14 & 43 & 0.00{\ldots}22\mathbf{1}11\ldots & 386 & 78 & 0.00{\ldots}00\mathbf{1}12\ldots & 1293 \\
9 & 0.000000000000020\mathbf{1}02\ldots & 16 & 44 & 0.00{\ldots}00\mathbf{1}12\ldots & 402 & 79 & 0.00{\ldots}00\mathbf{1}21\ldots & 1327 \\
10 & 0.0000000000000000\mathbf{1}00\ldots & 17 & 45 & 0.00{\ldots}00\mathbf{1}12\ldots & 421 & 80 & 0.00{\ldots}00\mathbf{1}02\ldots & 1361 \\
11 & 0.00{\ldots}22\mathbf{1}22\ldots & 24 & 46 & 0.00{\ldots}00\mathbf{1}00\ldots & 440 & 81 & 0.00{\ldots}00\mathbf{1}20\ldots & 1396 \\
12 & 0.00{\ldots}00\mathbf{1}12\ldots & 26 & 47 & 0.00{\ldots}00\mathbf{1}01\ldots & 460 & 82 & 0.00{\ldots}00\mathbf{1}10\ldots & 1431 \\
13 & 0.00{\ldots}00\mathbf{1}20\ldots & 31 & 48 & 0.00{\ldots}02\mathbf{1}20\ldots & 482 & 83 & 0.00{\ldots}22\mathbf{1}12\ldots & 1475 \\
14 & 0.00{\ldots}00\mathbf{1}00\ldots & 36 & 49 & 0.00{\ldots}00\mathbf{1}01\ldots & 501 & 84 & 0.00{\ldots}00\mathbf{1}22\ldots & 1503 \\
15 & 0.00{\ldots}00\mathbf{1}02\ldots & 42 & 50 & 0.00{\ldots}22\mathbf{1}22\ldots & 525 & 85 & 0.00{\ldots}00\mathbf{1}01\ldots & 1539 \\
16 & 0.00{\ldots}22\mathbf{1}22\ldots & 51 & 51 & 0.00{\ldots}00\mathbf{1}11\ldots & 544 & 86 & 0.00{\ldots}00\mathbf{1}02\ldots & 1576 \\
17 & 0.00{\ldots}00\mathbf{1}02\ldots & 55 & 52 & 0.00{\ldots}00\mathbf{1}10\ldots & 566 & 87 & 0.00{\ldots}02\mathbf{1}02\ldots & 1615 \\
18 & 0.00{\ldots}00\mathbf{1}01\ldots & 62 & 53 & 0.00{\ldots}02\mathbf{1}11\ldots & 590 & 88 & 0.00{\ldots}22\mathbf{1}12\ldots & 1656 \\
19 & 0.00{\ldots}00\mathbf{1}20\ldots & 70 & 54 & 0.00{\ldots}20\mathbf{1}11\ldots & 615 & 89 & 0.00{\ldots}02\mathbf{1}01\ldots & 1693 \\
20 & 0.00{\ldots}00\mathbf{1}20\ldots & 78 & 55 & 0.00{\ldots}00\mathbf{1}21\ldots & 635 & 90 & 0.00{\ldots}00\mathbf{1}00\ldots & 1728 \\
21 & 0.00{\ldots}00\mathbf{1}00\ldots & 86 & 56 & 0.00{\ldots}00\mathbf{1}02\ldots & 668 & 91 & 0.00{\ldots}22\mathbf{1}01\ldots & 1773 \\
22 & 0.00{\ldots}00\mathbf{1}01\ldots & 95 & 57 & 0.00{\ldots}00\mathbf{1}20\ldots & 683 & 92 & 0.00{\ldots}00\mathbf{1}11\ldots & 1807 \\
23 & 0.00{\ldots}22\mathbf{1}01\ldots & 110 & 58 & 0.00{\ldots}02\mathbf{1}12\ldots & 709 & 93 & 0.00{\ldots}00\mathbf{1}11\ldots & 1847 \\
24 & 0.00{\ldots}22\mathbf{1}22\ldots & 118 & 59 & 0.00{\ldots}02\mathbf{1}02\ldots & 738 & 94 & 0.00{\ldots}00\mathbf{1}01\ldots & 1895 \\
25 & 0.00{\ldots}22\mathbf{1}02\ldots & 131 & 60 & 0.00{\ldots}00\mathbf{1}02\ldots & 758 & 95 & 0.00{\ldots}00\mathbf{1}00\ldots & 1928 \\
26 & 0.00{\ldots}00\mathbf{1}01\ldots & 135 & 61 & 0.00{\ldots}00\mathbf{1}02\ldots & 784 & 96 & 0.00{\ldots}20\mathbf{1}21\ldots & 1972 \\
27 & 0.00{\ldots}00\mathbf{1}00\ldots & 146 & 62 & 0.00{\ldots}02\mathbf{1}00\ldots & 812 & 97 & 0.00{\ldots}00\mathbf{1}02\ldots & 2018 \\
28 & 0.00{\ldots}00\mathbf{1}20\ldots & 158 & 63 & 0.00{\ldots}20\mathbf{1}22\ldots & 841 & 98 & 0.00{\ldots}02\mathbf{1}00\ldots & 2055 \\
29 & 0.00{\ldots}00\mathbf{1}20\ldots & 170 & 64 & 0.00{\ldots}00\mathbf{1}22\ldots & 865 & 99 & 0.00{\ldots}00\mathbf{1}01\ldots & 2096 \\
30 & 0.00{\ldots}00\mathbf{1}01\ldots & 182 & 65 & 0.00{\ldots}02\mathbf{1}21\ldots & 894 & 100 & 0.00{\ldots}00\mathbf{1}00\ldots & 2139 \\
31 & 0.00{\ldots}00\mathbf{1}10\ldots & 195 & 66 & 0.00{\ldots}20\mathbf{1}10\ldots & 923 & 101 & 0.00{\ldots}00\mathbf{1}21\ldots & 2183 \\
32 & 0.00{\ldots}22\mathbf{1}02\ldots & 212 & 67 & 0.00{\ldots}00\mathbf{1}00\ldots & 949 & 102 & 0.00{\ldots}00\mathbf{1}22\ldots & 2227 \\
33 & 0.00{\ldots}00\mathbf{1}02\ldots & 222 & 68 & 0.00{\ldots}00\mathbf{1}00\ldots & 978 & 103 & 0.00{\ldots}00\mathbf{1}02\ldots & 2271 \\
34 & 0.00{\ldots}00\mathbf{1}01\ldots & 236 & 69 & 0.00{\ldots}00\mathbf{1}21\ldots & 1008 & 104 & 0.00{\ldots}00\mathbf{1}12\ldots & 2316 \\
35 & 0.00{\ldots}00\mathbf{1}20\ldots & 251 & 70 & 0.00{\ldots}00\mathbf{1}22\ldots & 1038 & 105 & 0.00{\ldots}00\mathbf{1}01\ldots & 2361 \\
\end{array}
\]
}
\vspace*{12pt}
\caption{Ternary verification for Fibonacci products \(a_n=\prod_{k=1}^n F_k\), \(1\le n\le 105\).
The column ``\(1\)st pos.''\  gives the position of the first ternary digit equal to \(1\); a dash indicates that \(\frac{1}{a_n}\in C\).}
\label{tab:fibonacci-products}
\end{table}

\begin{remark}\label{rem:fibonacci-margin}
The comparison \(\frac{5}{6}>\log_2\varphi\) leaves a positive but relatively narrow linear margin.
Accordingly, we observe that once the additive constants \(\log_2(c_{m,D})\) and the overhead \(t_2(m)\) are included, the explicit cutoff \(n_0\) furnished by \cref{cor:fixed-prime-criterion-largest-prime} may be large in concrete instances.

Indeed, the argument in the case of reciprocals of Fibonacci products suggests a broader (and genuinely marginal) class for the criterion.
Let \((u_k)_{k\ge 1}\) be a non-degenerate integer linear recurrence with a dominant root \(\lambda\) (so \(|u_k|\asymp |\lambda|^k\)).
Then every prime divisor of \(\prod_{k=1}^n u_k\) is at most \(\max_{k\le n}|u_k|\ll |\lambda|^n\), giving a prime-factor bound \(\beta(n)\ll |\lambda|^n\) and hence the right side of \eqref{eq:largest-prime-cutoff} grows like \(n\log_{p_0}|\lambda|\).
If one can find a fixed prime \(p_0\nmid m\) and an explicit linear lower bound \(\sum_{k=1}^n \vpp{p_0}(u_k)\ge cn\) for some \(c>0\), then \cref{cor:fixed-prime-criterion-largest-prime} can apply precisely when the number-theoretic inequality
\[
c>\log_{p_0}|\lambda|
\]
holds (up to the additive constants \(t(m,p_0)\) and \(\log_{p_0}(c_{m,D})\)).
Depending on the recurrence and the choice of \(p_0\), the comparison can go either way, making products of linear-recurrence values a natural boundary regime for the method.
\end{remark}

\subsection{Products of \(m^k-1\)}\label{sec:structural-example}

The preceding four applications all fit within the domain of the largest-prime-factor corollary. We now give a structural example showing that the full \cref{thm:fixed-prime-criterion-structural} adds non-trivial additional power: In the case of an increasing-degree product, \cref{thm:fixed-prime-criterion-structural} can apply because the radical-order term is straightforward to control, while the naive largest-prime-factor comparison has the wrong growth rate.

\begin{proposition}[Reciprocals of products of \(m^k-1\)]\label{prop:mk-minus-one}
Fix \(m\ge 3\) and a missing-digit set \(K_{m,D}\), and take
\[
a_n=\prod_{k=1}^n (m^k-1).
\]
Then, the intersection \(\left\{\frac{1}{a_n}:n\in\N\right\}\cap K_{m,D}\) is finite.
\end{proposition}

\begin{proof}
Since \((m,m^k-1)=1\) for every \(k\), we have \(Q_n=a_n\).

Now, let \(p_0\) be any odd prime divisor of \(m^2+1\). (Such a prime exists because \(m^2+1>2\) and is not a power of \(2\) for \(m\ge 3\).)
Because \(p_0\mid m^2+1\), we have \(m^2\equiv -1\pmod{p_0}\), so \(\ord_{p_0}(m)=4\).
In particular,
\[
p_0\mid m^k-1 \iff 4\mid k.
\]
Hence, we have
\[
\vpp{p_0}(Q_n)
=
\sum_{k=1}^n \vpp{p_0}(m^k-1)
\ge
\#\{1\le k\le n:4\mid k\}
=
\sum_{k=1}^n \mathbf{1}_{4\mid k}
=
\left\lfloor\frac{n}{4}\right\rfloor.
\]

Now let \(p\mid Q_n\). Then \(p\mid m^k-1\) for some \(k\le n\), so \(\ord_p(m)\mid k\).
Therefore
\[
\ord_{\rad(Q_n)}(m)\mid \lcm(1,2,\dots,n),
\]
and hence
\[
\vpp{p_0}\bigl(\ord_{\rad(Q_n)}(m)\bigr)
\le
\vpp{p_0}\bigl(\lcm(1,2,\dots,n)\bigr)
=
\left\lfloor \log_{p_0} n\right\rfloor.
\]

Thus, \cref{thm:fixed-prime-criterion-structural} applies with
\[
\alpha(n)=\left\lfloor\frac{n}{4}\right\rfloor,
\qquad
\gamma(n)=\left\lfloor \log_{p_0} n\right\rfloor.
\qedhere
\]
\end{proof}

As with our other applications, we can make \cref{prop:mk-minus-one} completely explicit for the standard Cantor set.

\begin{corollary}[Reciprocals of products of \(3^k-1\) in the standard Cantor set]
For \(C=K_{3,\{0,2\}}\) and
\[
a_n=\prod_{k=1}^n (3^k-1),
\]
we have
\[
\left\{\frac{1}{a_n}:n\in\N\right\}\cap C
=
\emptyset.
\]
\end{corollary}

\begin{proof}
We have \(m=3\) and \(D=\{0,2\}\), so \(c_{3,D}=4\).
Since \(3^2+1=10=2\cdot 5\), we take \(p_0=5\) as in the proof of \cref{prop:mk-minus-one}.
Then \(\ord_5(3)=4\) and
\[
t=\vpp{5}(3^4-1)=\vpp{5}(80)=1.
\]
\Cref{thm:fixed-prime-criterion-structural} applies with
\[
\alpha(n)=\left\lfloor\frac{n}{4}\right\rfloor,
\qquad
\gamma(n)=\left\lfloor\log_5(n)\right\rfloor.
\]

For \(12\le n\le 24\), we have \(\lfloor \frac{n}{4}\rfloor\ge 3\) and \(\lfloor \log_5(n)\rfloor=1\), so
\[
\left\lfloor\frac{n}{4}\right\rfloor
\ge 3
>
1+1+\log_5 4.
\]
For \(n\ge 25\), define
\[
g(x)\defeq \frac{x}{4}-2-\log_5 x-\log_5 4.
\]
Then
\[
g'(x)=\frac{1}{4}-\frac{1}{x\ln 5}>0
\qquad (x\ge 25),
\]
so \(g\) is increasing on \([25,\infty)\). Since
\[
g(25)=\frac{25}{4}-2-\log_5 25-\log_5 4
=\frac{25}{4}-4-\log_5 4
=2.25-\log_5 4>0,
\]
we see that \(g(n)>0\) for all \(n\ge 25\); using the inequalities \(\lfloor \frac{n}{4}\rfloor\ge \frac{n}{4}-1\) and \(\lfloor \log_5(n)\rfloor\le \log_5(n)\), this implies that
\[
\left\lfloor\frac{n}{4}\right\rfloor
>
1+\left\lfloor\log_5(n)\right\rfloor+\log_5 4
\qquad (n\ge 25).
\]
Therefore, \eqref{eq:structural-cutoff} holds for every \(n\ge 12\), and hence
\[
\frac{1}{a_n}\notin C
\qquad
(n\ge 12).
\]

It remains to inspect the cases \(1\le n\le 11\).
Since \(\gcd(3,3^k-1)=1\) for every \(k\), no \(a_n\) is divisible by \(3\), so every \(\frac{1}{a_n}\) lies outside the terminating regime and \cref{lem:exact-base-m-test} applies directly. Applying \cref{lem:exact-base-m-test} to each of the eleven reciprocals in question shows that every ternary expansion contains the digit \(1\); see \cref{tab:mk-minus-one}. Thus the intersection is empty.
\end{proof}

\begin{table}[tp]
\centering

\renewcommand{\arraystretch}{1.15}
\begin{tabular}{rlr}
\toprule
$n$ & base-$3$ prefix of $\tfrac{1}{a_n}$ & first $1$ pos.\\
\midrule
$1$ & $0.\mathbf{1}11\ldots$ & $1$ \\
$2$ & $0.00\mathbf{1}20\ldots$ & $3$ \\
$3$ & $0.00000\mathbf{1}20\ldots$ & $6$ \\
$4$ & $0.000000000\mathbf{1}20\ldots$ & $10$ \\
$5$ & $0.00000000000000\mathbf{1}21\ldots$ & $15$ \\
$6$ & $0.00{\ldots}00\mathbf{1}21\ldots$ & $21$ \\
$7$ & $0.00{\ldots}00\mathbf{1}21\ldots$ & $28$ \\
$8$ & $0.00{\ldots}00\mathbf{1}21\ldots$ & $36$ \\
$9$ & $0.00{\ldots}00\mathbf{1}21\ldots$ & $45$ \\
$10$ & $0.00{\ldots}00\mathbf{1}21\ldots$ & $55$ \\
$11$ & $0.00{\ldots}00\mathbf{1}21\ldots$ & $66$ \\
\bottomrule
\end{tabular}

\vspace*{12pt}
\caption{Ternary verification for products \(a_n=\prod_{k=1}^n(3^k-1)\), \(1\le n\le 11\).
Column conventions are as in \cref{tab:factorials}.}
\label{tab:mk-minus-one}
\end{table}

\begin{remark}\label{rem:mk-minus-one-structural-vs-largest-prime}
The family in \cref{prop:mk-minus-one} shows that the full structural criterion \cref{thm:fixed-prime-criterion-structural} can be materially easier to verify than the coarse largest-prime-factor packaging \cref{cor:fixed-prime-criterion-largest-prime}. Indeed, the trivial bound
\[
\Pplus(Q_n)\le m^n-1
\]
would make the right side of \eqref{eq:largest-prime-cutoff} grow linearly like \(n\log_{p_0}m\), whereas here
\[
\nu_{p_0}\bigl(\ord_{\rad(Q_n)}(m)\bigr)\le \lfloor \log_{p_0}n\rfloor
\]
is only logarithmic in \(n\). For example, when \(m=3\) and \(p_0=5\), the naive \(P^+\)-packaging has linear coefficient \(\log_5 3\approx 0.683\), while the lower bound in \cref{prop:mk-minus-one} has coefficient \(\frac{1}{4}\); thus, this route cannot recover \cref{prop:mk-minus-one} from the displayed estimates.

What drives the difference is that the prime-size information and the order information are on very different scales in this example. A prime divisor \(p\mid m^k-1\) may itself be as large as \(m^k-1\), but its contribution to the order structure is constrained by \(\ord_p(m)\mid k\). After taking the least common multiple over \(k\le n\), this yields an order bound through \(\lcm(1,\dots,n)\), which is logarithmic at the \(p_0\)-adic level even though the individual prime factors may be exponentially large.
\end{remark}

\section{Discussion}

\Cref{thm:fixed-prime-criterion-structural} isolates a specific reusable arithmetic pattern building on the approach of~\cite{LWY}. The imported ingredients are the comparison between \(\ord_Q(m)\) and \(\ord_{\rad(Q)}(m)\) extracted from Korobov's digit-distribution theorem, and the standard order-lifting facts that force \(\ord_Q(m)\) to acquire \(p_0\)-power divisibility when \(Q\) carries a large power of a fixed prime \(p_0\nmid m\). The theorem abstracts these into a denominator obstruction measured by
\[
\vpp{p_0}\bigl(\ord_{\rad(Q)}(m)\bigr),
\]
i.e., the \(p_0\)-adic cost contributed by the radical of the \(m\)-coprime denominator; this radical-order term is the natural invariant because it is exactly what remains after the Korobov comparison and the order-lifting lower bound are combined.

The applications in \cref{sec:applications} illustrate two levels of use. In the first four families, an estimate in terms of \(P^+(Q_n)\) already suffices. The family \(\prod_{k=1}^n(m^k-1)\) serves a different role: it shows that the radical-order term can remain logarithmic even when the most straightforward \(P^+\)-bound is exponential.

We close by noting a few ``non-examples'' illustrating two kinds of limitations of the fixed-prime method:

\begin{itemize}
\item For the \textbf{primorial sequence} \(a_n=p_n^\#=\prod_{p\le p_n} p\) (where \(p_n\) is the \(n\)-th prime and the product runs over primes), the criterion in \cref{thm:fixed-prime-criterion-structural} cannot apply for any fixed prime \(p_0\nmid m\). Indeed, \(a_n\) is squarefree, and removing the \(m\)-part preserves squarefreeness, so \(Q_n\) is squarefree as well. Hence for any fixed \(p_0\),
\[
\vpp{p_0}(Q_n)\le 1
\qquad (n\ge 1).
\]
Therefore any admissible lower bound in \cref{thm:fixed-prime-criterion-structural} must satisfy \(\alpha(n)\le 1\) on every tail. On the other hand, the structural cutoff demands
\[
\alpha(n)>t+\gamma(n)+\log_{p_0}(c_{m,D})
\]
for all sufficiently large \(n\), where necessarily \(\gamma(n)\ge 0\), \(t\ge 1\), and \(c_{m,D}\ge 4\). Thus the right side is always \(>1\).

\item Our fixed-prime criterion also cannot be used to establish finiteness for the \textbf{central binomial coefficients} \(a_n=\binom{2n}{n}\). Indeed, since \(p_0\nmid m\), removing the \(m\)-part does not change the \(p_0\)-adic valuation:
\[
\vpp{p_0}(Q_n)=\vpp{p_0}\left(\binom{2n}{n}\right).
\]
By Kummer's theorem \cite{kummer1852erganzungssatze}, \(\left(\vpp{p_0}\binom{2n}{n}\right)\) equals the number of carries when adding \(n+n\) in base \(p_0\). In particular, along the subsequence \(n=p_0^k\), we have
\[
\vpp{p_0}\left(\binom{2n}{n}\right)
=
\begin{cases}
0 & p_0 \text{ odd},\\
1 & p_0=2.
\end{cases}
\]
Hence, any admissible lower bound \(\alpha(n)\le \vpp{p_0}(Q_n)\) is upper-bounded by \(1\) along the infinite subsequence \(n=p_0^k\). But again the structural cutoff requires
\[
\alpha(n)>t+\gamma(n)+\log_{p_0}(c_{m,D}),
\]
with \(\gamma(n)\ge 0\), \(t\ge 1\), and \(c_{m,D}\ge 4\), so the right side is always \(>1\).
\end{itemize}

For primorials, squarefreeness prevents any fixed-prime valuation from growing beyond \(1\), so the structural cutoff can never be met. Central binomial coefficients fail for a different reason: along an infinite subsequence, every fixed \(p_0\)-adic valuation stays bounded by \(1\).

The family \(\prod_{k=1}^n(m^k-1)\) suggests that future refinements should be formulated first at the structural level, since the decisive order term can remain logarithmic even when the largest prime factor may be large. It would be interesting to develop variants of the criterion that allow a controlled choice of prime depending on \(n\), or to combine Korobov's estimate with additional structure in the factorization of \(a_n\). Even in cases where \cref{thm:fixed-prime-criterion-structural} fails, Korobov-type digit-distribution constraints may still interact fruitfully with other Diophantine inputs---for instance, primitive divisor results for recurrence sequences, or lower bounds for large prime factors of structured products.

\providecommand{\bysame}{\leavevmode\hbox to3em{\hrulefill}\thinspace}
\providecommand{\MR}{\relax\ifhmode\unskip\space\fi MR }
\providecommand{\MRhref}[2]{%
  \href{http://www.ams.org/mathscinet-getitem?mr=#1}{#2}
}
\providecommand{\href}[2]{#2}

\newpage
\appendix
\section{Computational verification}\label{app:computational-verification}

The following Python script carries out the finite verifications underlying the explicit intersection results in \cref{sec:applications} and generates Tables~\ref{tab:factorials}--\ref{tab:mk-minus-one}. For each sequence family, the cutoff \(N_0\) beyond which \cref{thm:fixed-prime-criterion-structural} (or \cref{cor:fixed-prime-criterion-largest-prime}) excludes membership is proved in the main text and supplied to the script as a constant; the program then tests every index \(n<N_0\) by performing the base-\(m\) long division of \(\frac{1}{a_n}\) via the remainder recurrence of \cref{lem:exact-base-m-test}.

\medskip
\begin{lstlisting}[language=Python]
#!/usr/bin/env python3
# -*- coding: utf-8 -*-
r"""
cantor-intersection-verification.py

Computational appendix for "A Fixed-Prime Criterion for\\Reciprocals in Missing-Digit Sets"

For each sequence family in Section 4, this script tests every
1 <= n < N_0 exactly (Lemma 4.1) and emits a LaTeX table (booktabs)
suitable for \input{}.

All arithmetic is exact (Python arbitrary-precision integers).
Requires Python >= 3.9; standard library only.

Specialised to the middle-third Cantor set C = K_{3,{0,2}}.

Paper parameters
----------------
  m = 3,  D = {0,2},  c_{m,D} = 4.
  p_0 = 2 ->  t(3,2) = v_2(3^2 - 1) - 1 = 2.       [Lemma 2.4]
  p_0 = 5 ->  t(3,5) = v_5(3^4 - 1) = v_5(80) = 1.  [Lemma 2.1]

Cutoffs N_0 (proved in Section 4)
---------------------------------
  Factorials          1/n!                 : N_0 = 10
  Superfactorials     1/prod_{k<=n} k!     : N_0 = 5
  Polynomial products 1/prod_{k<=n}(k^2+1) : N_0 = 30
  Fibonacci products  1/prod_{k<=n} F_k    : N_0 = 106
  Products of 3^k-1   1/prod_{k<=n}(3^k-1) : N_0 = 12
  
"""

from __future__ import annotations

import math
import sys
from typing import AbstractSet, List, Optional, Tuple

# ================================================================
#  Constants
# ================================================================

M: int = 3                                     # base
D_SET: frozenset[int] = frozenset({0, 2})      # allowed digits

N0_FACT:  int = 10     # factorials
N0_SFACT: int = 5      # superfactorials
N0_POLY:  int = 30     # polynomial products  (f = x^2 + 1)
N0_FIB:   int = 106    # Fibonacci products
N0_3K:    int = 12     # products of 3^k - 1

# Display parameters for abbreviated non-member prefixes
HEAD  = 2    # leading digits shown after the decimal point
TAIL  = 2    # digits shown immediately before the first bad digit
AFTER = 2    # digits shown immediately after the first bad digit


# ================================================================
#  Arithmetic helpers
# ================================================================

def _prime_factors(n: int) -> List[int]:
    """Distinct prime factors of  n >= 2,  in increasing order."""
    factors: List[int] = []
    d = 2
    while d * d <= n:
        if n % d == 0:
            factors.append(d)
            while n % d == 0:
                n //= d
        d += 1
    if n > 1:
        factors.append(n)
    return factors


# ================================================================
#  Generalised membership test  (arbitrary K_{m, D})
# ================================================================

def generalized_membership(a_n: int, m: int,
                           allowed: AbstractSet[int]
                           ) -> Tuple[bool, Optional[int]]:
    r"""
    Decide whether  1/a_n \in K_{m, D}  where  D = *allowed digits*.

    For terminating base-m expansions both the terminating and the
    non-terminating representations are checked, so the function
    is correct for composite bases as well as prime ones.

    Parameters
    ----------
    a_n     : positive integer
    m       : base (integer >= 2)
    allowed : set of allowed digits, each in {0, 1, ..., m-1}

    Returns
    -------
    (True,  None)  if  1/a_n \in K_{m, D}.
    (False, pos )  otherwise;  *pos*  (1-indexed) is the position
                   of the first disallowed digit.
    """
    if a_n <= 0:
        raise ValueError(f"a_n must be positive, got {a_n}")
    if m < 2:
        raise ValueError(f"base must be >= 2, got {m}")

    # 1/1 = 0.(m-1)(m-1)(m-1)...
    if a_n == 1:
        return (True, None) if (m - 1) in allowed else (False, 1)

    # --- Does 1/a_n terminate in base m? ---
    # Terminates iff every prime factor of a_n also divides m.
    reduced = a_n
    for p in _prime_factors(m):
        while reduced % p == 0:
            reduced //= p
    can_terminate = (reduced == 1)

    if can_terminate:
        # ------ Terminating expansion (finitely many digits) ------
        r: int = 1
        digits: List[int] = []
        first_one: Optional[int] = None
        while r != 0:
            q, r = divmod(m * r, a_n)
            digits.append(q)
            if q not in allowed and first_one is None:
                first_one = len(digits)

        # Terminating form:  d_1 ... d_L  0 0 0 ...
        if first_one is None and 0 in allowed:
            return True, None

        # Non-terminating form:  d_1 ... d_{k-1}  (d_k - 1)  (m-1)(m-1)...
        last_nz = -1
        for i in range(len(digits) - 1, -1, -1):
            if digits[i] != 0:
                last_nz = i
                break
        if last_nz >= 0 and (m - 1) in allowed:
            alt_d = digits[last_nz] - 1
            if (all(digits[i] in allowed for i in range(last_nz))
                    and alt_d in allowed):
                return True, None

        # Neither representation works.
        term_fb = first_one if first_one is not None else len(digits) + 1
        nt_fb: Optional[int] = None
        if last_nz >= 0:
            for i in range(last_nz):
                if digits[i] not in allowed:
                    nt_fb = i + 1
                    break
            if nt_fb is None:
                if digits[last_nz] - 1 not in allowed:
                    nt_fb = last_nz + 1
                elif (m - 1) not in allowed:
                    nt_fb = last_nz + 2
        candidates = [c for c in [term_fb, nt_fb] if c is not None]
        return False, min(candidates)

    else:
        # ------ Non-terminating expansion (unique representation) ------
        # Exit immediately on first disallowed digit.
        # Confirm membership only when the remainder cycle closes.
        r = 1
        seen: set = set()
        pos = 0
        while True:
            if r in seen:
                return True, None
            seen.add(r)
            q, r = divmod(m * r, a_n)
            pos += 1
            if q not in allowed:
                return False, pos


def cantor_membership(a_n: int) -> Tuple[bool, Optional[int]]:
    r"""Decide whether  1/a_n \in C = K_{3,\{0,2\}}."""
    return generalized_membership(a_n, M, D_SET)


# ================================================================
#  LaTeX display helpers
# ================================================================

def _base_m_prefix(a_n: int, m: int, allowed: AbstractSet[int],
                   first_one: Optional[int],
                   max_show: int = 20) -> str:
    r"""
    Short LaTeX string for the initial base-m digits of  1/a_n.

    Non-members: when the first disallowed digit is near the
    start, all digits through a short context window are shown.
    When it is far away, an abbreviated form is used:

        0.ab{\ldots}cd\mathbf{1}ef\ldots

    where  ab  are the first HEAD digits,  cd  are the TAIL digits
    immediately preceding the bold disallowed digit, and  ef  are
    the AFTER digits immediately following it.

    Members: compact  \overline{}  notation.
    Assumes  m <= 10  for single-character digit display.
    """
    # --- a_n = 1 ---
    if a_n == 1:
        d = str(m - 1)
        if first_one is None:
            return r"0.\overline{" + d + "}"
        return r"0.\mathbf{" + d + "}"

    # --- Non-member: compute the needed digits ---
    if first_one is not None:
        need = first_one + AFTER
        r_val = 1
        digits: List[int] = []
        for _ in range(need):
            q, r_val = divmod(m * r_val, a_n)
            digits.append(q)
            if r_val == 0:
                break

        idx = first_one - 1                      # 0-indexed
        end = min(first_one + AFTER, len(digits))

        if end <= max_show:
            # Short form: show every digit through the context window
            parts = [str(d) for d in digits[:end]]
            parts[idx] = rf"\mathbf{{{digits[idx]}}}"
            return "0." + "".join(parts) + r"\ldots"

        # Abbreviated form:
        #   head_digits {\ldots} tail_digits \mathbf{d} after_digits \ldots
        h = min(HEAD, idx)
        hs = "".join(str(digits[i]) for i in range(h))
        ts = max(h, idx - TAIL)
        tl = "".join(str(digits[i]) for i in range(ts, idx))
        bd = rf"\mathbf{{{digits[idx]}}}"
        af = "".join(str(digits[i]) for i in range(idx + 1, end))
        return "0." + hs + r"{\ldots}" + tl + bd + af + r"\ldots"

    # --- Member: detect cycle for compact display ---
    reduced = a_n
    for p in _prime_factors(m):
        while reduced % p == 0:
            reduced //= p

    if reduced == 1:
        # Terminating -> show non-terminating form
        r_val = 1
        digs: List[int] = []
        while r_val != 0:
            q, r_val = divmod(m * r_val, a_n)
            digs.append(q)
        last_nz = len(digs) - 1
        while last_nz >= 0 and digs[last_nz] == 0:
            last_nz -= 1
        if last_nz >= 0:
            pre = "".join(str(d) for d in digs[:last_nz])
            return rf"0.{pre}{digs[last_nz] - 1}\overline{{{m - 1}}}"
        return r"0.\overline{" + str(m - 1) + "}"

    # Non-terminating: find cycle
    r_val = 1
    seen_r: dict[int, int] = {}
    digs = []
    for i in range(max_show + 100):
        if r_val in seen_r:
            cs = seen_r[r_val]
            pre = "".join(str(d) for d in digs[:cs])
            rep = "".join(str(d) for d in digs[cs:])
            if pre:
                return rf"0.{pre}\overline{{{rep}}}"
            return rf"0.\overline{{{rep}}}"
        seen_r[r_val] = i
        q, r_val = divmod(m * r_val, a_n)
        digs.append(q)

    return "0." + "".join(str(d) for d in digs[:max_show]) + r"\ldots"


def _ternary_prefix(a_n: int, first_one: Optional[int],
                    max_show: int = 20) -> str:
    """LaTeX string for the initial ternary digits of  1/a_n."""
    return _base_m_prefix(a_n, M, D_SET, first_one, max_show)


# ================================================================
#  Table generation and file output
# ================================================================

Row = Tuple[int, bool, Optional[int], str]


def build_latex(title: str, rows: List[Row]) -> str:
    """Build a LaTeX tabular string from the given rows."""
    lines = [
        f"% {title}",
        r"\renewcommand{\arraystretch}{1.15}",
        r"\begin{tabular}{rlr}",
        r"\toprule",
        (r"$n$ & base-$3$ prefix of $\tfrac{1}{a_n}$"
         r" & first $1$ pos.\\"),
        r"\midrule",
    ]
    for n, member, pos, prefix in rows:
        p = r"---" if member else f"${pos}$"
        lines.append(rf"${n}$ & ${prefix}$ & {p} \\")
    lines.append(r"\bottomrule")
    lines.append(r"\end{tabular}")
    return "\n".join(lines)


def emit(title: str, rows: List[Row], filename: str) -> None:
    """Print a summary and table to stdout; write .tex file."""
    latex   = build_latex(title, rows)
    members = [n for n, m, _, _ in rows if m]
    n0      = rows[-1][0] + 1 if rows else 1

    print("=" * 64)
    print(f"  {title}")
    print(f"  Cutoff N_0 = {n0}  |  Members: {members}")
    print("=" * 64)
    print(latex)
    print()

    with open(filename, "w") as f:
        f.write(f"% Auto-generated by cantor_verify.py\n")
        f.write(f"% {title}\n")
        f.write(f"% Include via:  \\[\\input{{{filename}}}\\]\n")
        f.write(f"% For tables exceeding one page, wrap in a "
                f"longtable environment.\n\n")
        f.write(latex + "\n")

    print(f"  => written to {filename}\n\n")


def emit_multicolumn(title: str, rows: List[Row], filename: str,
                     cols: int = 3) -> None:
    """Write a multi-column array for long tables."""
    members = [n for n, m, _, _ in rows if m]
    n0 = rows[-1][0] + 1 if rows else 1

    # Split rows into roughly equal chunks
    k, r = divmod(len(rows), cols)
    chunks: List[List[Row]] = []
    start = 0
    for i in range(cols):
        end = start + k + (1 if i < r else 0)
        chunks.append(rows[start:end])
        start = end

    max_len = max(len(c) for c in chunks)

    # Column spec: rlr separated by inter-group spacing
    col_spec = " @{\\hspace{1.5em}\\vrule width 0.4pt\\hspace{1.5em}} ".join(["rlr"] * cols)

    lines = [
        f"% Auto-generated by cantor_verify.py",
        f"% {title}",
        r"\renewcommand{\arraystretch}{1.05}",
        r"\[",
        rf"\begin{{array}}{{{col_spec}}}",
    ]

    # Header row
    hdr = r"n & \text{prefix of $1/a_n$} & \text{1st pos.}"
    lines.append(" & ".join([hdr] * cols) + r" \\")
    lines.append(r"\hline")

    # Data rows
    for row_idx in range(max_len):
        parts: List[str] = []
        for chunk in chunks:
            if row_idx < len(chunk):
                n, member, pos, prefix = chunk[row_idx]
                p = r"\text{--}" if member else str(pos)
                parts.append(rf"{n} & {prefix} & {p}")
            else:
                parts.append("& &")
        lines.append(" & ".join(parts) + r" \\")

    lines.append(r"\end{array}")
    lines.append(r"\]")

    content = "\n".join(lines)

    print("=" * 64)
    print(f"  {title}")
    print(f"  Cutoff N_0 = {n0}  |  Members: {members}")
    print("=" * 64)

    with open(filename, "w") as f:
        f.write(content + "\n")
    print(f"\n  => written to {filename}\n\n")


# ================================================================
#  Sequence values
# ================================================================

def vals_factorial(nmax: int) -> List[int]:
    """a_n = n!  for n = 0, 1, ..., nmax."""
    v = [1] * (nmax + 1)
    for n in range(1, nmax + 1):
        v[n] = v[n - 1] * n
    return v


def vals_superfactorial(nmax: int) -> List[int]:
    """a_n = prod_{k=1}^n k!  for n = 0, ..., nmax."""
    v = [1] * (nmax + 1)
    f = 1
    for n in range(1, nmax + 1):
        f *= n
        v[n] = v[n - 1] * f
    return v


def vals_poly_x2plus1(nmax: int) -> List[int]:
    """a_n = prod_{k=1}^n (k^2 + 1)  for n = 0, ..., nmax."""
    v = [1] * (nmax + 1)
    for n in range(1, nmax + 1):
        v[n] = v[n - 1] * (n * n + 1)
    return v


def vals_fibonacci_product(nmax: int) -> List[int]:
    """a_n = prod_{k=1}^n F_k  for n = 0, ..., nmax."""
    fib = [0] * (max(nmax + 1, 3))
    fib[1] = fib[2] = 1
    for k in range(3, nmax + 1):
        fib[k] = fib[k - 1] + fib[k - 2]
    v = [1] * (nmax + 1)
    for n in range(1, nmax + 1):
        v[n] = v[n - 1] * fib[n]
    return v


def vals_mk_minus_1(nmax: int, m: int = 3) -> List[int]:
    """a_n = prod_{k=1}^n (m^k - 1)  for n = 0, ..., nmax."""
    v = [1] * (nmax + 1)
    for k in range(1, nmax + 1):
        v[k] = v[k - 1] * (m ** k - 1)
    return v


# ================================================================
#  Driver
# ================================================================

def run(title: str, n0: int, vals: List[int], fname: str) -> None:
    """Test n = 1, ..., n0 - 1 and emit the result table."""
    rows: List[Row] = []
    for n in range(1, n0):
        A = vals[n]
        mem, pos = cantor_membership(A)
        prefix = _ternary_prefix(A, pos)
        rows.append((n, mem, pos, prefix))
    emit(title, rows, fname)


def main() -> None:
    # (a) Factorials: a_n = n!
    run(f"Factorials: a_n = n! [N_0 = {N0_FACT}]",
        N0_FACT, vals_factorial(N0_FACT - 1), "table_factorial.tex")

    # (b) Superfactorials: a_n = prod_{k=1}^n k!
    run(f"Superfactorials: a_n = prod k! [N_0 = {N0_SFACT}]",
        N0_SFACT, vals_superfactorial(N0_SFACT - 1),
        "table_superfactorial.tex")

    # (c) Polynomial products: a_n = prod_{k=1}^n (k^2 + 1)
    run(f"Polynomial products: a_n = prod (k^2+1) [N_0 = {N0_POLY}]",
        N0_POLY, vals_poly_x2plus1(N0_POLY - 1),
        "table_poly_x2plus1.tex")

    # (d) Fibonacci products -- use multi-column layout (105 rows)
    n0 = N0_FIB
    vals = vals_fibonacci_product(n0 - 1)
    rows: List[Row] = []
    for n in range(1, n0):
        A = vals[n]
        mem, pos = cantor_membership(A)
        prefix = _ternary_prefix(A, pos)
        rows.append((n, mem, pos, prefix))
    emit_multicolumn(
        f"Fibonacci products: a_n = prod F_k [N_0 = {n0}]",
        rows, "table_fibonacci.tex", cols=3)

    # (e) Products of m^k - 1: a_n = prod_{k=1}^n (3^k - 1)
    run(f"Products of 3^k - 1 [N_0 = {N0_3K}]",
        N0_3K, vals_mk_minus_1(N0_3K - 1), "table_mk_minus_1.tex")


if __name__ == "__main__":
    main()
\end{lstlisting}

\end{document}